\documentclass[11pt]{amsart}
\usepackage{amsmath,amstext,amssymb,amsopn,amsthm}

\usepackage[top=25mm,bottom=25mm,left=25mm,right=25mm]{geometry}

\makeatletter
\@namedef{subjclassname@2020}{\textup{2020} Mathematics Subject Classification}
\makeatother
\newtheorem{lemma}{Lemma}[section]
\newtheorem{proposition}[lemma]{Proposition}
\newtheorem{theorem}[lemma]{Theorem}
\newtheorem{corollary}[lemma]{Corollary}
\newtheorem{definition}[lemma]{Definition}
\newtheorem{example}[lemma]{Example}

\theoremstyle{definition}
\newtheorem{remark}[lemma]{Remark}
\usepackage{hyperref}

\makeatletter
\@addtoreset{equation}{section}
\makeatother

\newcommand{\mpr}{\mathbb P}
\newcommand{\mr}{{\mathbb R}}
\newcommand{\R}{{\mathbb R}}

\newcommand{\Rd}{{\mr^d}}
\newcommand{\dx}{{\rm d}x}
\newcommand{\dy}{{\rm d}y}
\newcommand{\dz}{{\rm d}z}
\newcommand{\dt}{{\rm d}t}
\DeclareMathOperator{\dist}{dist}
\newcommand{\dw}{{\rm d}w}

\newcommand*\oU{\overline{U}}
\newcommand*\pord{\int_{\mR^d}}

\newcommand*\mE{\mathbb{E}}
\newcommand*\mR{\mathbb{R}}
\newcommand*\me{\mathrm{e}}

\newcommand{\form}[3]{\ifx\\#3\\
						\mathcal{#1}_{#2}
					\else
						\mathcal{#1}_{#2}^{(#3)}
					\fi}
\newcommand{\spc}[3]{\ifx\\#3\\
							\mathcal{#1}_{#2}
						\else
							\mathcal{#1}_{#2}^{\, #3}
						\fi}
\newcommand*\epd{\form{E}{D}{p}}
\newcommand*\hpd{\form{H}{D}{p}}
\newcommand*\vd{\spc{V}{D}{}}
\newcommand*\wdp{\spc{W}{D}{p}}
\newcommand*\ydp{\spc{Y}{D}{p}}
\newcommand*\vdp{\spc{V}{D}{p}}
\newcommand*\xdp{\spc{X}{D}{p}}
\DeclareMathOperator{\sgn}{sgn}

\newcommand*\DN{{D\! N}}

\usepackage{color}

\usepackage{color}
\usepackage{xcolor}

\newcommand*{\ar}{\textcolor{black}}

\title{
Nonlinear nonlocal
Douglas identity}
\author[K.~Bogdan]{Krzysztof Bogdan}
\address{Faculty of Pure and Applied Mathematics, Wroc\l{}aw University of Science and Technology, Wyb. Wyspia\'nskiego 27, 50-370 Wroc\l{}aw, Poland.}
\email{krzysztof.bogdan@pwr.edu.pl}
\author[T.~Grzywny]{Tomasz Grzywny}
\address{Faculty of Pure and Applied Mathematics, Wroc\l{}aw University of Science and Technology, Wyb. Wyspia\'nskiego 27, 50-370 Wroc\l{}aw, Poland.}
\email{tomasz.grzywny@pwr.edu.pl}
\author[K.~Pietruska-Pa\l uba]{Katarzyna Pietruska-Pa\l uba}
\address{Institute of Mathematics, University of Warsaw, ul. Banacha 2, 02-097 Warsaw, Poland.}
\email{kpp@mimuw.edu.pl}
\thanks{K.B. and K.P.-P. were supported by the NCN grant 2018/31/B/ST1/03818. T.G. was supported by the NCN grant 2017/27/B/ST1/01339. A.R. was supported by the NCN grant 2015/18/E/ST1/00239.}
\author[A.~Rutkowski]{Artur Rutkowski}
\address{Faculty of Pure and Applied Mathematics, Wroc\l{}aw University of Science and Technology, Wyb. Wyspia\'nskiego 27, 50-370 Wroc\l{}aw, Poland.}
\email{artur.rutkowski@pwr.edu.pl}

\begin{document}
\date{\today}
\begin{abstract}
We give Hardy--Stein and Douglas identities
for nonlinear nonlocal Sobolev--Bregman integral forms with unimodal L\'evy measures. We prove that the corresponding Poisson integral defines an extension operator for the Sobolev--Bregman spaces. \ar{As an application, we obtain the boundedness of the Dirichlet-to-Neumann operator on weighted $L^p$ spaces.}
We also show that the Poisson integrals
 are quasiminimizers of the Sobolev--Bregman forms.
\end{abstract}
\maketitle
\section{Introduction}
In 1931 J. Douglas \cite{MR1501590}
established a connection of the energy of the
harmonic function $u$ on the unit disc $B(0,1)$ with the ``energy'' of its boundary trace $g$, regarded as a function on $[0,2\pi)$:
\begin{equation}\label{eq:Douglas-classical}
\int_{B(0,1)}|\nabla u(x)|^2 {\rm d}x = \frac{1}{8\pi} \iint_{[0,2\pi)\times[0,2\pi)}\frac{(g(\eta)-g(\xi))^2}{\sin^2((\eta-\xi)/2)}{\rm d}\eta{\rm d}\xi.
\end{equation}
 The formula arose in the study of the so-called Plateau problem --- the problem of existence of minimal surfaces posed by J.-L. Lagrange.
The identity holds true provided that
the left-hand side is finite --- for details see, e.g., Chen and Fukushima \cite[(2.2.60)]{MR2849840}.
Thus, under the integrability condition, \eqref{eq:Douglas-classical}
is valid
for the
solutions of the Dirichlet problem,
	\[\begin{cases}
	\Delta u=0&{\rm in}\ B(0,1),\\
	u=g&{\rm in}\ \partial B(0,1).\end{cases}\]

In our paper we propose a variant of \eqref{eq:Douglas-classical}, which we call nonlinear nonlocal Douglas identity.
The term ``nonlocal''
means that the Laplace operator $\Delta$ above is replaced by a nonlocal operator $L.$
Specifically, we adopt the following setting.
Let $d=1,2,\ldots$. Suppose that the function $\nu\colon [0,\infty)\to (0,\infty]$ is nonincreasing and, with a slight abuse of notation, let $\nu(z)=\nu(|z|)$ for $z\in \Rd$.
In particular, $\nu$ is symmetric, i.e., $\nu(z)=\nu(-z)$, $z\in \mR^d$. Assume further that
\begin{equation}\label{eq:Lmc}
\int_{\R^d} \nu(z)\,\dz=\infty \quad \mbox{and}\quad
\int_{\Rd} \left(|z|^2\wedge 1\right)\nu(z)\,\dz<\infty.
\end{equation}
Thus, $\nu$ is a strictly positive density function of an infinite isotropic unimodal L\'{e}vy measure on $\Rd$ (in short, $\nu$ is unimodal).
For $u\colon \Rd\to \R$ and $x\in \Rd$ we let
\begin{eqnarray}\label{eq:L-def}
Lu(x)
&=&\lim_{\epsilon\to 0^+} \int_{|x-y|>\epsilon} (u(y)-u(x))\nu(x,y)\,\dy\\
&=&\lim_{\epsilon\to 0^+} \tfrac12 \!\int_{|z|>\epsilon} \!\!\!(u(x+z)+u(x-z)-2u(x))\nu(z)\,\dz.\nonumber
\end{eqnarray}
Here, $\nu(x,y):=\nu(y-x)$, and the limit exists, e.g., for
$u$ in $C_c^\infty (\mathbb R^d)$, the smooth functions with compact support.
Operators of the form \eqref{eq:L-def} are called nonlocal, because the value of $L \phi(x)$ also depends on the values of $\phi$ outside of a neighborhood of $x$.
Furthermore, the operators satisfy the maximum principle, meaning that if $\phi(x_0)=\sup\{\phi(x): x\in \Rd\}$, then $L\phi(x_0)\le 0$. It is well known that such operators may be used to describe transportation of mass, charge, etc. in elliptic and parabolic equations; especially to pose boundary-value problems.

To our nonlocal setting we bring a judicious way of measuring the smoothness of functions
for a given set. Let $D\subset\Rd$ be open.
For the sake of gradual introduction we first consider the quadratic form
\begin{equation}\label{eq:SVf}
\form{E}{D}{}[u]=\;\;\;\;\tfrac{1}{2}\!\!\!\!\!\!\!\!\!\iint\limits_{\R^d\times\R^d\setminus D^c\times D^c}
(u(x)-u(y))^2\nu(x,y)\,\dx\dy.
\end{equation}
Such forms appeared in Servadei and Valdinoci \cite{MR2879266, MR3002745},
where the set $\R^d\times\R^d\setminus D^c\times D^c$ was denoted $Q$, then in Ros-Oton \cite[(3.1)]{MR3447732} and Dipierro, Ros-Oton and Valdinoci \cite[p. 379]{MR3651008}.
Similar forms were also used in Felsinger, Kassmann and Voigt \cite[Definition 2.1 (ii)]{MR3318251}.
$\form{E}{D}{}$ is the energy functional of the nonlocal Dirichlet problem
\begin{equation}\label{eq:Dp}
\left\{
\begin{array}{ll}
Lu=0 & \mbox{ in } D,\\
u=g & \mbox{ on } D^c,
\end{array}\right.
\end{equation}
see \cite{MR2879266, MR3002745} and
Bogdan, Grzywny, Pietruska-Pa\l{}uba and Rutkowski \cite{MR4088505}.
It should be noted that $\form{E}{D}{}$ is better than the vanilla form $\form{E}{\mR^d}{}$ for solving \eqref{eq:Dp}, because it
allows for more general
external conditions $g$ due to the restriction of integration in \eqref{eq:SVf} to $Q=\R^d\times\R^d\setminus D^c\times D^c$, cf. \cite[p. 39]{MR4088505}. Therefore $\form{E}{D}{}$ constitutes an important step forward
in nonlocal variational problems;
we refer the reader to \cite{MR4088505} for more details and to \cite{MR2879266, MR3002745} for applications to nonlinear equations. We note that our results
also have consequences for the Dirichlet problem for $L$ on $D$ when $\form{E}{\mR^d}{}$ is used, see Corollary~\ref{cor:nnDia} below.

Numerous papers study the nonlocal
Dirichlet problem by variational methods for nonlocal
operators --- in the present setting we should note \cite{MR3318251}, \cite{MR3447732}, and Rutkowski \cite{MR3738190}.
It is known for many L\'evy and L\'evy-type kernels $\nu$ and bounded $D$ \cite{MR3318251, MR3738190}, \cite[Section 5]{MR4088505} that a unique weak solution of \eqref{eq:Dp} exists provided that $g\colon D^c\to \R$ can be extended to a function $u\in L^2(D)$ from the Sobolev class
\begin{equation}\label{eq:dSs}
\vd
:=\{u\colon \R^d\to \R\ |\ \form{E}{D}{} [u]<\infty\}.
\end{equation}
It is therefore important to determine conditions on $g$ that allow for such an extension
--- in other words --- to determine the {\em trace space}, say, $\spc{X}{D}{}$, of $\vd$. We note in passing that by \cite[Lemma 3.4]{MR4088505},
the functions from $\vd$ are automatically square integrable on $D$.
For the fractional Laplacian $\Delta^{\alpha/2}:=-(-\Delta)^{\alpha/2}$ (see Subsection~\ref{sec:pot-theo-notions} for a definition) a solution to this problem was proposed by Dyda and Kassmann \cite{DYDA2018108134} by using the Whitney decomposition and the method of reflection. In fact, \cite[Theorem~3]{DYDA2018108134}
concerns general \textit{$p$-increments}, i.e., $|u(x)-u(y)|^p$ with $p\ge 1$.

In
\cite{MR4088505} we resolved the extension and trace problem
for $p=2$ for a wide class of unimodal L\'evy operators
by a different approach
based on the (quadratic) nonlocal Douglas identity.
Namely, \cite[Theorem 2.3]{MR4088505} asserts that the trace space $\spc{X}{D}{}$
consists of functions $g\colon D^c\to D$ for which the following
form on $D^c$ is finite,
 \[\form{H}{D}{}[g]:= \tfrac 12\iint\limits_{D^c\times D^c} (g(z)-g(w))^2\gamma_D(z,w)\,\dz\dw.\]
Here and afterwards we call
\begin{align*}
\gamma_ D(w,z) &
:= \int_ D \int_ D \nu(w,x)G_ D(x,y) \nu(y,z)\,\dx\dy,\quad w,z\in D^c,
\end{align*}
the {\em kernel of interaction via $D$}, or {\em interaction kernel}, and $G_D$ is the Green function of $L$ for $D$;
see Section \ref{sec:pot-theo-notions} for details. \ar{We note that $\gamma_D$ is the nonlocal normal derivative  of the Poisson kernel of $L$, see \eqref{e.dnd} and \eqref{eq:rg2} below, similarly as the kernel in the classical Douglas identity, see Bogdan, Fafuła, and Rutkowski \cite[Subsection~2.3]{2022arXiv220707431B}.}
The nonlocal Douglas identity of \cite{MR4088505} can be stated as follows,
\begin{equation}\label{eq:ident-l2}
\form{E}{D}{}[u]= \form{H}{D}{}[g],
\end{equation}
where $g\colon D^c\to \R$, $\form{H}{D}{}[g]<\infty$, and $u= P_D[g]$ is the Poisson integral of $g$, see Section \ref{sec:pot-theo-notions}. Notably, $P_D[g]$ is a harmonic function of $L$, so
the identity \eqref{eq:ident-l2} explains the energy of a harmonic function by the energy of its external values.
In the language of Chen and Fukushima \cite[Chapter 5]{MR2849840}, the right-hand side of \eqref{eq:ident-l2} is the \textit{trace form} and $\gamma_D(z,w)\,\dz\dw$ is the \textit{Feller measure} for $(\form{E}{\mR^d}{},\spc{V}{\mR^d}{})$ on $D^c$, but the extension and trace problem for $\vd$ were not investigated in \cite{MR2849840}.
\ar{We also note that Jacob and Schilling \cite{MR1681637} studied Douglas identities for nonlocal censored-type Dirichlet forms.}

Our present goal is to extend the nonlocal Douglas formula \eqref{eq:ident-l2} to a more general nonlinear case. The possibility of such a setting
occurred to us
owing to the recent Hardy--Stein identities of Bogdan, Dyda and Luks \cite[Theorem 2]{MR3251822}.
To this end we will use the following notion, the \textit{French power}:
$$x^{\langle\kappa \rangle}=|x|^\kappa \sgn(x),\quad x\in \R,\, \quad\kappa\in \R.$$
More precisely, $x^{\langle\kappa\rangle}=x^\kappa$ if $x>0$, $x^{\langle\kappa\rangle}=-|x|^\kappa$ if $x<0$, and $0^{\langle\kappa\rangle}=0$. For example, $x^{\langle 0\rangle}=\sgn(x)$ and $x^{\langle 2\rangle}\neq x^2$ as functions on $\R$.
In what follows we fix $1<p<\infty$, the exponent of the ``nonlinearity'' alluded to in the title of the paper.
Our nonlinear nonlocal Douglas identity is as follows:
\begin{align}\label{eq:nsDi}
\nonumber
&&\tfrac12\!\!\!\!\!\iint\limits_{\R^d\times \R^d \setminus D^c \times D^c} (u(x)^{\langle p-1\rangle}-
u(y)^{\langle p-1\rangle})(u(x)-u(y))\,\nu(x,y)\,\dx\dy\\
&&=\tfrac12
\iint\limits_{D^c\times D^c}
(g(w)^{\langle p-1\rangle}-
g(z)^{\langle p-1\rangle})(g(w)-g(z))
\gamma_D(w,z)\,\dw\dz,
\end{align}
where $u=P_D[g]$ and $g\colon D^c\to \R$. For a precise statement see Theorem \ref{th:main} and Remark~\ref{rem:pf1.8} below, since
the result hinges on suitable additional assumptions on $\nu$, $D$ and $g$.
No analogue of \eqref{eq:nsDi} seems to exist in the literature for $p\neq 2$, even for $\Delta^{\alpha/2}$. However, related nonlinear forms $\int u^{\langle p-1\rangle}Lu$ appear often in the literature concerning Markovian semigroups of operators on $L^p$ spaces, see also \eqref{eq:rp2} and \eqref{e:fag}
below. This is because for $p\in (1,\infty)$ the dual space of $L^p$ is $L^{p/(p-1)}$ and for $u\in L^p$ we have $u^{\langle p-1\rangle}\in L^{p/(p-1)}$, and $\int |u|^p=\int |u^{\langle p-1\rangle}|^{p/(p-1)}=\int u^{\langle p-1\rangle}u$. Therefore in view of the Lumer--Phillips theorem, $u^{\langle p-1\rangle}$ yields a linear functional on $L^p$ appropriate for testing \textit{dissipativity} of generators, see, e.g., Pazy \cite[Section 1.4]{MR710486}. In this connection we note that Davies \cite[Chapter 2 and 3]{MR1103113} gives some fundamental calculations with forms and powers. For the semigroups generated by local operators we refer to Langer and Maz'ya \cite{MR1694522} and Sobol and Vogt \cite[Theorem 1.1]{MR1923627}. Liskevich and Semenov \cite{MR1409835} use the $L^p$ setting to analyze perturbations of Markovian semigroups.
For nonlocal operators we refer to Farkas, Jacob and Schilling \cite[(2.4)]{MR1808433}, and to the monograph of Jacob \cite[(4.294)]{MR1873235}.

\ar{The following variant of \eqref{eq:nsDi} is also true, see \eqref{eq:defEpD}, \eqref{eq:Ep}, and \eqref{eq:sHp} below,
\begin{align}
\nonumber
&\tfrac12\!\!\!\!\!\iint\limits_{\R^d\times \R^d \setminus D^c \times D^c} (|u(y)|^p - |u(x)|^p - pu(x)^{\langle p-1\rangle}(u(y) - u(x)))\,\nu(x,y)\,\dx\dy\\
\label{e.nvDi}
&=\tfrac12
\iint\limits_{D^c\times D^c}
(|g(z)|^p - |g(w)|^p - pg(w)^{\langle p-1\rangle}(g(z) - g(w)))
\gamma_D(w,z)\,\dw\dz.
\end{align}
The integrands in \eqref{e.nvDi} come from the second order Taylor remainder of the convex function $x\mapsto |x|^p$, see \eqref{eq:Fpdef}, which leads us to the notion  of \textit{Bregman
	divergence}; see Subsection~\ref{sec:fp}, see also Bregman \cite{MR215617} for the original contribution or Sprung \cite{MR3897043}.   Bregman divergence is important for statistical learning, see Nielsen and Nock \cite{MR2598413} or Frigyik, Gupta and Srivastava \cite{MR2589887} and the references therein.  The Bregman divergence based on the power function $|x|^p$ defines the free energy functionals in the studies of Sobolev and Gagliardo--Nirenberg--Sobolev inequalities by Carrillo et al. \cite[p. 71]{MR1853037} and Bonforte, Dolbeault, Nazaret, and Simonov \cite{2020arXiv200703674B}. It also commonly appears in entropy inequalities, see, e.g., Wang \cite{MR3206685}.}

\ar{The present paper indicates further uses of Bregman divergence in PDEs. As we show in Section~\ref{sec:app}, $\gamma_D$ is the kernel of the Dirichlet-to-Neumann map \eqref{eq:dtndef}
for $L$. 
Over the last few years, Dirichlet-to-Neumann map related to nonlocal operators was intensively studied in the context of inverse problems, see, e.g., \cite{MR4078233, MR4083776, MR4237942, MR4383014}.
The forms in \eqref{eq:nsDi} are suitable for studying the Dirichlet-to-Neumann map as an operator in $L^p$. In particular, using our Douglas identity we show that the \emph{normalized} Dirichlet-to-Neumann operator \eqref{e.nDNo}  is bounded on a certain weighted $L^p$ space. Results in this direction were obtained by Vondra\v{c}ek \cite{MR4245573} and Foghem and Kassmann \cite{2022arXiv220406793F} for $p=2$, but even in this case our approach gives new insights.}

\ar{As another motivation, we mention that the form on the left-hand side of \eqref{eq:nsDi} with $D=\mR^d$ is appropriate for studying $L^p$ properties of Markovian semigroups. For instance, it was used by Bogdan, Jakubowski, Lenczewska, and Pietruska-Pa\l{}uba \cite{MR4372148} to characterize the  contractivity on $L^p(\mR^d)$ of the semigroups generated by the fractional Laplacian with Hardy-type potentials. The interested reader may find insights into the technique in \cite[Lemma 7 and Proof of Theorem 3]{MR4372148}, or even in \eqref{eq:rp2} and \eqref{e:fag} below.
}

The paper is organized as follows. Section \ref{sec:prelim} contains definitions and basic facts. Subsection \ref{sec:pot-theo-notions} introduces notions from the probabilistic potential theory and Subsection \ref{sec:fp} introduces our nonlinear setting and
novel Sobolev--Bregman spaces $\vdp$ and $\xdp$ defined by the condition of finiteness of the respective sides of \eqref{eq:nsDi}.  In \eqref{e.cFH} we collect in one place four (equivalent) approximations for our Bregman divergence, which appear in the literature.
In Section \ref{sec:harm} we generalize the Hardy--Stein identities of \cite{MR3251822} and \cite{MR4088505} to our present context. This is instrumental for the proof of the Douglas identity
in Section \ref{sec:mainthm}.
In Corollary~\ref{coro:extension} we conclude that the Poisson integral $P_D$ and the restriction to $D^c$ are the extension and trace operators between the Sobolev--Bregman
spaces.
In view toward applications in variational problems, in Section~\ref{sec:aug} we
prove the Douglas formula with the remainder for the energy of sufficiently regular
nonharmonic functions. We also show that harmonic functions are quasi-minimizers of the considered nonlinear nonlocal forms, but in general not
minimizers. \ar{In Section~\ref{sec:app} we apply our results for the analysis of the Dirichlet-to-Neumann operator in $L^p$ for $p\geq 2$.
Finally,} in
Section~\ref{sec:discussion} we give, for $p\ge 2$,
the following result for Poisson integrals $u=P_D[g]$ and the
more usual integral forms based on the
$p$-increments of functions:
	\begin{equation}
\iint\limits_{\R^d\times\R^d\setminus D^c\times D^c} |u(x)-u(y)|^p\nu(x,y)\,\dx\dy
\leq c
\iint\limits_{D^c\times D^c}|g(w)-g(z)|^p\gamma_D(w,z)\,
\dw\dz\,.\label{eq:th61}
\end{equation}
It follows
that $g\mapsto P_D[g]$ is an extension operator
for nonlocal Sobolev-type spaces $\wdp$, defined
by the finiteness of the left-hand side.
In the remainder of Section \ref{sec:discussion} we
compare
$\vdp$ and $\wdp$.

\textbf{Acknowledgments.} We thank Tomasz Adamowicz, W\l{}odzimierz B\c{a}k, Artur Bogdan, Bart\l{}omiej Dyda, Agnieszka Ka\l{}amajska, Moritz Kassmann, Mateusz Kwa\'snicki, Ren\'e Schilling and Enrico Valdinoci for discussions, comments or suggestions.

\section{Preliminaries}\label{sec:prelim}
All the considered functions, sets and measures are tacitly assumed to be Borel.
When we write $f\approx g$ (resp. $f
\lesssim g$), we mean that there is a number $c>0$, i.e. a \textit{constant}, such that $(1/c) f(x)\leq g(x)\leq c f(x)$ (resp. $f(x)
\leq c g(x)$) for all arguments $x$. Important constants will be capitalized: $C_1,C_2,\ldots,$ and their values will not change throughout the paper.
\subsection{Processes and potential-theoretic notions}\label{sec:pot-theo-notions}

Let
$L$ and $\nu$ be as in the Introduction. Following \cite{MR4088505}, we \textit{additionally} assume that:
\begin{itemize}
	\item[\bf (A1)]
	$\nu$ is twice continuously differentiable on $(0,\infty)$ and there is a constant $C_1$ such that
	\begin{equation*}\label{e:on}
		|\nu'(r)|, |\nu''(r)|\leq C_1\nu(r), \qquad r>1.
	\end{equation*}
	\item[\bf (A2)] There exist constants $\beta\in (0,2)$ and $C_2>0$ such that		
	\begin{eqnarray}\label{eq:nuSc}\nu(\lambda r)&\leq& C_2 \lambda^{-d-\beta}\nu(r) ,\qquad 0<\lambda, r\leq 1,\\ \label{es:nuSc1}
		\nu(r)&\leq& C_2 \nu(r+1), \qquad\quad r\geq 1.
	\end{eqnarray}
\end{itemize}
A prominent representative of \textit{unimodal} L\'evy operators $L$
is the fractional Laplacian $\Delta^{\alpha/2}:=-(-\Delta)^{\alpha/2}$. In this case we have $\nu(x,y) = c_{d,\alpha}|y-x|^{-d-\alpha}$, where $\alpha \in (0,2)$, $x,y\in \Rd$, and $$c_{d,\alpha} = \frac{2^\alpha \Gamma((d+\alpha)/2)}{\pi^{d/2}|\Gamma(-\alpha/2)|}.$$
We refer the reader to Bogdan and Byczkowski \cite{MR1671973}, Di Nezza, Palatucci and Valdinoci \cite{MR2944369}, Garofalo \cite{MR3916700}, and Kwa\'snicki \cite{MR3613319} for more information on $\Delta^{\alpha/2}$.
Clearly,
$\nu(r)=c_{d,\alpha}r^{-d-\alpha}$ satisfies both
\textbf{(A1)} and \textbf{(A2)}.

Our results depend in part on martingale properties of harmonic functions, so
we introduce
the L\'{e}vy process $(X_t,t\ge 0)$
on $\mathbb R^d$ whose generator
is given by \eqref{eq:L-def}.
Let
$$
\psi(\xi)=\int_\Rd (1-\cos \xi\cdot x)\nu(|x|)\,\dx, \quad \xi \in \Rd,
$$
the L\'evy--Khinchine exponent of $(X_t)$. Since $\nu(\mR^d) = \infty$, by Sato \cite[Theorem 27.7]{MR1739520} and Kulczycki and Ryznar \cite[Lemma 2.5]{MR3413864},
the densities $p_t(x)$ of $(X_t)$ are continuous on $\mR^d\setminus \{0\}$ for $t>0$, and satisfy
$$
\int_\Rd \me^{i\xi\cdot x} p_t(x)\,\dx=\me^{-t\psi(\xi)}, \quad t>0,\,\xi\in \Rd.
$$
For $t>0$ and $x,\,y\in\Rd$ denote $p_t(x,y)=p_t(y-x)$, the transition density of $(X_t)$ considered as Markov process on $\Rd$. Namely, for starting point $x\in \Rd$, times $0\leq t_1<t_2<\ldots t_n$ and sets $A_1, A_2,\ldots A_n\subset \Rd$ we let, as usual,
$$\mpr^x(X_{t_1}\in A_1,\ldots,X_{t_n}\in A_n)\!=\!\!
\int\limits_{A_1}\!\int\limits_{A_2}\!\ldots\int\limits_{A_n} \!p_{t_1}(x,x_1) p_{t_2-t_1}(x_1,x_2)\cdots p_{t_n-t_{n-1}}(x_{n-1},x_n)\,\dx_1\dx_2\cdots\dx_n.$$
This determines $\mpr^x$, the distribution of the process $(X_t)$ starting from $x$, and $\mE^x$, the corresponding expectation.
In the wording of \cite[Section 11]{MR1739520}, $(X_t)$ is the symmetric L\'evy process in $\Rd$
with $(0,\nu,0)$ as the L\'evy triplet.
Without losing generality we actually assume
that each $X_t$ is the canonical projection $X_t(\omega)=\omega(t)$ on the space of c\`adl\`ag functions $\omega\colon [0,\infty)\to \Rd$. We will also use the standard complete right-continuous filtration $(\mathcal F_t, t\geq 0)$ to analyze $(X_t)$, see Protter \cite[Theorem I.31]{MR2273672}. In passing we recall that every L\'evy process is a Feller process \cite{MR2273672}.

Let $\emptyset \neq D\subset \R^d$ be an open set.
The time of the first exit of $X$ from $D$ is, as usual,
$$\tau_D=\inf\{t>0: \, X_t\notin D\}.$$
The Dirichlet heat kernel $p_t^D(x,y)$ is defined by Hunt's formula, cf. Chung and Zhao \cite[Chapter 2.2]{MR1329992},
$$
p_t^D(x,y) = p_t(x,y) - \mE^x(p_{t-\tau_D}(X_{\tau_D},y);\, \tau_D<t),\quad t>0,\ x,y\in\mR^d.
$$
It is the transition density of the process $(X_t)$ killed upon exiting $D$, i.e.,
$$\mE^x[t<\tau_D ;\, f(X_t)]=\int_\Rd f(y)p^D_t(x,y)\dy, \quad x\in \Rd,\, t>0\,,$$
for integrable functions $f$.
The Green function of $D$ is the potential of $p_t^D$:
$$
G_D(x,y)=\int_0^\infty p_t^D(x,y)\,\dt,\quad x,y\in\mR^d,
$$
and by Fubini--Tonelli we have
\begin{equation}\label{eq:gdtd}
\mE^x \tau_D=\int_\Rd G_D(x,y)\,\dy,\quad x\in\mR^d.
\end{equation}
The Poisson kernel of $D$ for $L$ is defined by
\begin{equation}\label{eq:Pk}
P_D(x,z)=\int_D G_D(x,y)\nu(y,z)\,\dy,\quad x\in D,\ z\in D^c.
\end{equation}
With \textbf{(A2)} for bounded set $D$ we easily see that
for all $x,y\in D$ and $z\in D^c$ with $\dist(z,
D
)\geq \rho>0$,
\begin{equation}\label{eq:aaaa}
\nu(x,z) \approx \nu(y,z),
 \end{equation}
where comparability constants depend on $\nu$, $D$ and $\rho$. Consequently, \eqref{eq:Pk} implies
\begin{equation}\label{eq:pjP}
P_D(x,z)\approx \nu(x,z)\mathbb{E}^x\tau_D,\quad x\in D, \quad \dist(z,D
)\geq \rho>0,
\end{equation} with the same proviso on comparability constants. Note that if $D$ is bounded and $x\in D$ is fixed, then $\mE^x\tau_D$ is bounded by a positive constant, see
Pruitt \cite{pruitt1981}.
We further note that for $w,z\in D^c$ the interaction kernel satisfies
\begin{align}
\gamma_ D(w,z) &= \int_ D \int_ D \nu(w,x)G_ D(x,y) \nu(y,z)\,\dx\dy\label{eq:rg1}\\
&= \int_ D\nu(w,x)P_D(x,z)\,\dx
= \int_D \nu(z,x)P_D(x,w)\,\dx = \gamma_D(z,w).\label{eq:rg2}
\end{align}
Finally, the $L$-harmonic measure of $D$ for $x\in \Rd$ is, as usual,
\begin{equation}\label{eq:hm}
\omega_D^x({\rm d}z) =\mpr^x[X_{\tau_D}\in {\rm d}z],
\end{equation}
the distribution of the random variable $X_{\tau_D}$ with respect to $\mathbb P^x.$

From the Ikeda--Watanabe formula (see, e.g., Bogdan, Rosi\'{n}ski, Serafin and Wojciechowski \cite[Section~4.2]{MR3737628}) it follows that $P_D(x,z)\,{\rm d}z$ is the part of $\omega_D^x({\rm d}z)$ which results from the discontinuous exit from $D$ (by a jump). Below, by suitable assumptions on $D$ and $\nu$, we assure that $P_D$ is the density of the whole harmonic measure, that is
\begin{equation}\label{eq:always-jump}
\int_{D^c}P_D(x,z)\,{\rm d}z=1,\quad x\in D.
\end{equation}
This is true, e.g., if $D$ is bounded, $\nu$ satisfies {\bf(A2)}, $|\partial D| = 0$ and $D^c$ has the property {\rm(VDC)}. The latter means that there is $c>0$ such that for every $r>0$ and $x\in \partial D$,
	\begin{equation}\label{eq:VDC}
	|D^c \cap B(x,r)| \geq cr^{d}.
	\end{equation}
	Here, as usual, $B(x,r)=\{y\in \Rd: |y-x|<r\}$.
For the proof of \eqref{eq:always-jump} under the above conditions,
see \cite[Corollary A.2]{MR4088505}.

Observe that for $U\subset D$ we have $p^U\le p^D$ and $G_U\le G_D$. Therefore,
$P_U(x,z)\le P_D(x,z)$ for $x\in U$, $z\in D^c$, and $\gamma_U(z,w)\leq \gamma_D(z,w)$ for $z,w\in D^c$.
These inequalities may be referred to as {\em domain monotonicity}.
For $g\colon D^c\to \R$ we define the Poisson extension of $g$:
\begin{eqnarray}\label{eq:exg}
	P_D[g](x)= \left\{
	\begin{array}{ll}
		g(x) &\mbox{ for } x\in D^c,\\[2mm]
		\int_{D^c} g(z)P_D(x,z)\,\dz & \mbox{ for }\ x\in D,
	\end{array}\right.
\end{eqnarray}
and we call
$\int_{D^c} g(z)P_D(x,z)\,\dz$ the Poisson integral, as long as it is convergent.

\subsection{Function $F_p$ and related function spaces}\label{sec:fp}
We depend on the two humble real functions:
$$x\mapsto |x|^\kappa \quad {\rm and }\quad x\mapsto x^{\langle \kappa\rangle},\quad x\in\mR, \quad \kappa\in \R.$$ Clearly, $|x|^\kappa$ is symmetric, $x^{\langle \kappa\rangle}$ is antisymmetric: $(-x)^{\langle \kappa\rangle}=-x^{\langle \kappa\rangle}$,
and their derivatives obey
$$(|x|^{\kappa})' = \kappa x^{\langle \kappa-1\rangle}\quad {\rm and }\quad (x^{\langle \kappa\rangle})' = \kappa|x|^{\kappa-1}, \quad x\neq 0.$$
Recall that $p>1$. We let
\begin{align}F_p(a,b)=|b|^p-|a|^p-p a^{\langle p-1\rangle}(b-a),\quad a,b \in \mathbb R.\label{eq:Fpdef}\end{align}
For instance, if $p=2$, then $F_2(a,b)=(b-a)^2$, and if $p=4$, then $F_4(a,b)=(b-a)^2(b^2+2ab+3a^2)$.
As the second-order Taylor remainder of the convex function $|x|^p$, $F_p$ is nonnegative. In fact,
\begin{equation}\label{elem-ineq}
F_p(a,b)
\;\approx\; (b-a)^2(|b|\vee |a|)^{p-2}, \quad
a, b\in\mathbb R,
\end{equation}
see \cite[Lemma~6]{MR3251822}.
In particular, for $p\geq 2$ we have
\begin{equation}\label{elem-ineq-2}
F_p(a,b)\approx
(b-a)^2(|a|^{p-2}+|b|^{p-2}),\quad a,b\in \R.
\end{equation}

Recall that if $X$ is a random variable with the first moment finite
and $a\in\R$, then
 \begin{equation}\label{eq:variance}
 \mathbb E(X-a)^2=\mathbb E(X-\mathbb E X)^2 + (\mathbb EX-a)^2=\mbox{Var}\,X+ (\mathbb EX-a)^2.
 \end{equation}
Here we do not exclude the case $\mathbb E X^2=\infty$, in which case both sides of \eqref{eq:variance} are infinite, hence equal.
This
variance formula has the following analogue for $F_p$.
\begin{lemma}\label{lem:aaa} Let $p> 1.$
Suppose that $X$ is a random variable such that $\mathbb E|X|<\infty$. Then,
\begin{enumerate}
\item[(i)]$\mathbb EF_p(\mathbb EX, X)=\mathbb E|X|^p-\left|\mathbb EX\right|^p\geq 0,$
\item[(ii)]$\mathbb E F_p(a,X)=F_p(a,\mathbb EX)+\mathbb EF_p(\mathbb EX,X)\geq \mathbb EF_p(\mathbb EX,X),\quad a\in\R,$
\item[(iii)]$\mathbb EF_p(a,X)= \mathbb EF_p(b, X) + F_p(a,b) +(pa^{\langle p-1\rangle}-pb^{\langle p-1\rangle})(b-\mathbb EX), \quad a,b\in \mR.$
\end{enumerate}
\end{lemma}
\begin{proof} The verification is elementary, but we present it
to emphasize that the finiteness of the first moment suffices. We have
\begin{align*}
\mE F_p(\mE X, X) = \mE\Big[|X|^{p} - |\mE X|^p - p(\mE X)^{\langle p-1\rangle}(X - \mE X)\Big] = \mE|X|^p - |\mE X|^{p},
\end{align*}
where $\mathbb E|X|^p=\infty$ is permitted, too. The expression in (i) is nonnegative by Jensen's inequality or because $F_p$ is nonnegative.
For all $a\in \mR$ we have,
\begin{align*}
\mE F_p(a,X) =& \,\mE \Big[|X|^p - |a|^p - pa^{\langle p-1\rangle}(X - a)\Big] \\
=&\, \mE \Big[|X|^p - |\mE X|^p - p(\mE X)^{\langle p-1\rangle}(X - \mE X)\Big]
 + |\mE X|^p - |a|^p - pa^{\langle p-1\rangle}(\mE X - a)\\
=&\, \mE F_p(\mE X, X) + F_p(a,\mE X)\geq \mE F_p(\mE X, X),
\end{align*}
as claimed in (ii). Finally, for all $a,b\in \mR$ the right-hand side of (iii) is
\begin{align*}
\mE |X|^p - |b|^p - pb^{\langle p-1\rangle}(\mE X - b)
+
|b|^p-|a|^p-p a^{\langle p-1\rangle}(b-a)
+
(pa^{\langle p-1\rangle}-pb^{\langle p-1\rangle})(b-\mathbb EX),
\end{align*}
which simplifies to the left-hand side of (iii). Needless to say, (ii) is a special case of (iii).
\end{proof}
We next propose a simple lemma concerning the $p$-th moments of
random variables, which is another generalization of \eqref{eq:variance}.
\begin{lemma}\label{lem:elem-p} For every $p\geq 1$
there exist constants $0<c_p\leq C_p$ such that for every random variable
$X$ with $\mathbb E|X|<\infty$ and every number $a\in \mathbb R,$
\begin{equation}\label{eq:p-variance}
c_p \left(\mathbb E|X-\mathbb E X|^p +|\mathbb EX-a|^p\right)
\leq \mathbb E|X-a|^p\leq C_p \left(
\mathbb E|X-\mathbb E X|^p +|\mathbb EX-a|^p\right).
\end{equation}
\end{lemma}
\begin{proof}
If $\mathbb E|X|^p=\infty$, then all the sides of \eqref{eq:p-variance} are infinite. Otherwise, by convexity,
\[\mathbb E|X-a|^p= \mathbb E|(X-\mathbb EX)+(\mathbb EX-a)|^p
\leq 2^{p-1} \left(\mathbb E|X-\mathbb EX|^p +\left|\mathbb EX-a\right|^p\right).\]
For the lower bound we make two observations:
$\left|\mathbb EX-a\right|^p\leq \mathbb E|X-a|^p$ (Jensen's inequality),
and
\begin{eqnarray*}
\mathbb E|X-\mathbb EX|^p &=& \mathbb E |(X-a)-(\mathbb EX-a)|^p
\leq 2^{p-1} \left(\mathbb E|X-a|^p +\left|\mathbb EX-a\right|^p\right)
\leq 2^p \mathbb E|X-a|^p.
\end{eqnarray*}
Adding the two, we get that
$|\mathbb EX-a|^p+\mathbb E|X-\mathbb EX|^p\leq (1+2^p)\mathbb E|X-a|^p.$
\end{proof}

The function $F_p(a,b)$ is not symmetric in $a,b$, but the right-hand side of \eqref{elem-ineq} is, so it is natural to consider the
symmetrized version of $F_p$, given by the formula:
\begin{eqnarray}\label{eq:wH}
H_p(a,b)& = & \tfrac 12(F_p(a,b)+F_p(b,a))
= \tfrac p2 (b^{\langle p-1\rangle} - a^{\langle p-1\rangle})(b-a),\quad a,b\in\mathbb R.
\end{eqnarray}
We can relate $H_p$ to a ``quadratic'' expression as follows.
\begin{lemma}\label{lem:quad}For every $p>1$ we have
	 $F_p(a,b)\approx H_p(a,b)\approx (b^{\langle p/2\rangle} - a^{\langle p/2\rangle})^2$.
\end{lemma}
\begin{proof}
The first comparison follows from
\eqref{elem-ineq}: we have $F_p(a,b)\approx F_p(b,a)$, hence $F_p\approx H_p$. As for the second statement, if either $a$ or $b$ are equal to $0$, then the expressions coincide up to constants depending on $p$.
	If $a,b\neq 0$, then $a = tb$ with $t\neq 0$. Using this representation we see that the second comparison
	is equivalent to the following:
	\begin{equation*}
	(t^{\langle p-1\rangle} - 1)(t-1) \approx (t^{\langle p/2\rangle}-1)^2,\quad t\in \mR.
	\end{equation*}
The latter holds
because both sides are continuous and positive except at $t=1$; at infinity
both are power functions with the leading term $|t|^p$, and at $t=1$ their ratio converges to a positive constant.
 \end{proof}
Summarizing, by \eqref{elem-ineq} and Lemma~\ref{lem:quad} for each $p\in (1,\infty)$ we have
\begin{equation}\label{e.cFH}
F_p(a,b)\approx H_p(a,b)\approx(b-a)^2(|b|\vee |a|)^{p-2}\approx (b^{\langle p/2\rangle} - a^{\langle p/2\rangle})^2, \quad
a, b\in\mathbb R.
\end{equation}
It
is hard to trace down the first occurrence of such comparisons
 in the literature. The one-sided inequality
$|b^{p/2}-a^{p/2}|^2\le \frac{p^2}{4(p-1)}(b-a)(b^{p-1}-a^{p-1})$ for $a,b\ge 0$ can be found in connection with logarithmic Sobolev inequalities, e.g., in Davies \cite[(2.2.9)]{MR1103113} for $2<p<\infty$, and Bakry \cite[p. 39]{MR1307413} for $p>1$. The opposite inequality
$(b-a)(b^{p-1}-a^{p-1})\leq (b^{p/2}-a^{p/2})^2$ with $a,b>0$ and $p>1$ appears, e.g., in \cite[Lemma 2.1]{MR1409835}.

In fact the following inequalities hold for all $p\in (1,\infty)$ and $a,b\in \R$:
\begin{equation}\label{e.cHk}
\frac{4(p-1)}{p^2}(b^{\langle p/2\rangle}-a^{\langle p/2\rangle})^2\leq (b-a)(b^{\langle p-1\rangle}-a^{\langle p-1\rangle})\leq 2(b^{\langle p/2\rangle}-a^{\langle p/2\rangle})^2.
\end{equation}
Indeed, if $a$ and $b$ have opposite signs then it is enough to consider $b=t\ge 1$ and $a=-1$, and to compare
$(t+1)(t^{p-1}+1)=t^p+t^{p-1}+t+1$ with $(t^{p/2}+1)^2=t^p+2t^{p/2}+1$.
We have $t^{p/2}= \sqrt{t^{p-1}t}\le (t^{p-1}+t)/2$, which verifies the left-hand side inequality in \eqref{e.cHk} with constant $1$, which is better than $4(p-1)/p^2$.
We further get the right-hand side inequality in \eqref{e.cHk}, and the constant $2$ suffices, because $t^{p-1}+t -(t^p+1)=(1-t)(t^{p-1}-1)\le 0$. Note that the constant $2$ is not optimal for individual values of $p$, e.g., for $p=2$, but the constant $1$ does not suffice for $p\in (1,2)\cup(2,\infty)$ because then $1\vee (p-1) > p/2$, and so $t^{p-1} + t > 2t^{p/2}$ for large $t$.

If $a$ and $b$ have the same sign, then we may assume $b=ta$, $a>0$, $t\geq 1$, and consider
the quotient
\[H(t)= \frac{(t^{p-1}-1)(t-1)}{(t^{p/2}-1)^2}= 1 -\frac{t(t^{(p-2)/2}-1)^2}{(t^{p/2}-1)^2} = 1- h(s)^2,
\]
where $s=\sqrt t$, $h(s)= s(s^{p-2}-1)/(s^p-1)$. We see that $h(s)$ is strictly positive for $p>2,$ $s>1$ and negative for $p\in (1,2)$. We claim that it decreases in the former case and increases in the latter. The sign of the derivative of $h$ is the same as the sign of the function $l(s) = -s^{2p-2} + (p-1)s^p - (p-1)s^{p-2} + 1$. Now, since $l(1) = 0$, the sign of $l$ on $(1,\infty)$ is in turn equal to the sign of $l'(s) = (p-1)s^{p-3}(-2s^p + ps^2 - (p-2))$, and further equal to the sign of $-2p(s^{p-1} - s)$. Since the last function is negative on $(1,\infty)$ if $p>2$ and positive for $p\in (1,2)$, the claim is proved. Consequently, the function $s\mapsto h(s)^{2}$ is decreasing on $(1,\infty)$, so we get
\[\lim_{t\to 1^+} H(t) = \frac{4(p-1)}{p^2} < H(t) < 1, \quad t> 1,\]
and \eqref{e.cHk} follows. The above also shows
that the constant $4(p-1)/p^2$ in \eqref{e.cHk} cannot be improved.

We would like to note that for $p\neq 2$, $F_p(a+t,b+t)$ is \textit{not} comparable with $F_p(a,b)$. Indeed, for $a,r>0$ one has $F_p(a,a+r)\approx r^2(a\vee (a+r))^{p-2} = r^2(a+r)^{p-2}$, which is not comparable with $F_p(0,r)=r^2$ for large values of $a$.
Here are one-sided comparisons of $F_p(a,b)$ with the more usual $p$-increments, \ar{see, e.g., Zeidler \cite[p. 503]{MR1033498}.}
\begin{lemma}\label{lem:p-incr}
If $p\geq 2$ then $F_p(a,b) \gtrsim |b-a|^p$,
and if $1<p\leq 2$, then $|b-a|^p\gtrsim F_p(a,b)$.
\end{lemma}
\begin{proof}
If $a=b$, then the inequalities are trivial, so assume that $a\neq b$ and consider the quotient
\[\frac{F_p(a,b)}{|b-a|^p}\approx \frac{(|a|\vee |b|)^{p-2}}{|b-a|^{p-2}}.\]
Both parts of the statements now follow from the inequality $|b-a|^r\leq 2^r(|a|\vee |b|)^r$, $r>0.$

\end{proof}

In analogy to \eqref{eq:SVf} for $u\colon\Rd\to \R$ we define
\begin{equation}\label{eq:defEpD}
\epd[u]:= \tfrac 1p\!\!\!\iint\limits_{\R^d\times \R^d\setminus D^c\times
	D^c}F_p(u(x),u(y))\nu(x,y)\,\dx\dy.
	\end{equation}
By the symmetry of $\nu$ and \eqref{eq:wH},
\begin{eqnarray}\epd[u]&=&
\tfrac1p\!\!\!\!\!\!\iint\limits_{\R^d\times \R^d\setminus D^c\times
D^c}H_p(u(x),u(y))\nu(x,y)\,\dx\dy \nonumber \\
&=&
\tfrac{1}{2}\!\!\!\!\!\!\iint\limits_{\R^d\times \R^d\setminus D^c\times
D^c}
(u(y)^{\langle p-1\rangle}-
u(x)^{\langle p-1\rangle})(u(y)-u(x))
\nu(x,y)\,\dx\dy.\label{eq:Ep}
\end{eqnarray}
Of course, $\form{E}{D}{2}=\form{E}{D}{}$.
For $D=\R^d$ we have
\begin{equation}\label{eq:fsp}
\form{E}{\mR^d}{p}[u]=
\tfrac 12\iint\limits_{\R^d\times \R^d} (u(y)^{\langle p-1\rangle}-
u(x)^{\langle p-1\rangle})(u(y)-u(x))\nu(x,y)\,\dx\dy.
\end{equation}
Clearly, for $p=2$ we retrieve the classical
Dirichlet form of the operator $L.$

Let $g\colon D^c\to \R$. To quantify the increments of $g$, we use the form:
\begin{eqnarray}\nonumber
\hpd[g] &=& \tfrac 1p\iint\limits_{D^c\times D^c}F_p(g(w),g(z))\gamma_D(w,z)\,\dw\dz
= \tfrac 1p\iint\limits_{D^c\times D^c}H_p(g(w),g(z))\gamma_D(w,z)\,\dw\dz\\
&=& \tfrac{1}{2} \iint\limits_{D^c\times D^c}(g(z)^{\langle p-1\rangle}-
g(w)^{\langle p-1\rangle})(g(z)-g(w))\gamma_D(w,z)\,\dw\dz.\label{eq:sHp}
\end{eqnarray}
The spaces $\vd$ and $\spc{X}{D}{}$ discussed in the Introduction lend themselves to the following generalizations:
\begin{equation}\label{eq:vdp}
\vdp:=\{u\colon \R^d\to\R\ |\ \epd[u]<\infty\},
\end{equation}
and
\begin{equation}\label{eq:xdp}
 \xdp:= \{g\colon D^c\to\R\ |\ \hpd[g]<\infty\}.
\end{equation}
We call them Sobolev--Bregman spaces, since they involve the Bregman divergence.
Our development below indicates that $\vdp$ and $\xdp$
provide a viable framework for
nonlocal nonlinear variational problems.
In view of \eqref{eq:Ep} for all $u\colon\Rd\to \R$ we have
\begin{equation}\label{eq:rp2}
\epd[u]=\form{E}{D}{}(u^{\langle p-1\rangle},u),
\end{equation}
where
\begin{equation*}
\form{E}{D}{}(v,u) := \tfrac 12 \iint\limits_{\mR^d\times \mR^d\setminus D^c\times D^c} (v(x) - v(y))(u(x) - u(y))\nu(x,y)\, \dx\dy,
\end{equation*}
if the integral is well defined, which is the case in \eqref{eq:rp2} for $v=u^{\langle p-1\rangle}$.
For clarity we also note that by \eqref{e.cHk}, \eqref{eq:Ep} and \eqref{eq:sHp}, we have the \textit{comparisons}
\begin{equation}\label{eq:apEp2}
\epd[u] \approx \form{E}{D}{}[u^{\langle p/2\rangle}],
\end{equation}
and
\begin{equation}\label{eq:apHp2}
\hpd[g] \approx \form{H}{D}{}[g^{\langle p/2\rangle}],
\end{equation}
for all $u\colon\Rd\to \R$ and $g\colon D^c\to \R$ with the comparability constants depending only on $p$.
Below, however, we focus on genuine \textit{equalities}.

\section{Hardy--Stein identity}\label{sec:harm}

We first collect properties of harmonic
functions that are needed in the proof of the identity
\eqref{eq:nsDi}. We mostly follow
\cite{MR4088505}, so our presentation will be brief.
We write $U\subset\subset D$ if the closure of $U$ is a compact subset of $D$.

\begin{definition}\label{def:rh}
	 We say that the function $u\colon \mathbb R^d\to\mathbb R$ is $L$-harmonic
(or harmonic, if $L$ is understood) in $D$ if it has the mean value property inside $D$, that is: for every open set $U\subset\subset D$,
$$u(x) = \mathbb{E}^x u(X_{\tau_U}), \quad x\in U.$$
If $u(x) = \mathbb{E}^x u(X_{\tau_D})$ for all $x\in D$, then we say that $u$ is regular harmonic.
 \end{definition}
 In the above we
assume that the expectations are absolutely convergent.

The strong Markov property of $(X_t)$ implies that if $u$ is regular $L$-harmonic in $D$, then it is $L$-harmonic in $D$.
By \cite[Section 4]{MR4088505}, if $u$ is
	$L$-harmonic in $D$, then $u\in L^1_{loc} (\mR^d) \cap C^2(D),$
$Lu(x)$ can be computed pointwise for $x\in D$ as in \eqref{eq:L-def}, and $Lu(x)=0$ for $x \in D$. We also note that the Harnack inequality holds for $L$-harmonic functions (see Grzywny and Kwa\'snicki \cite[Theorem 1.9]{MR3729529}; the assumptions of that theorem follow from {\bf (A2)}).

We will use the following Dynkin-type lemma, proven in our setting in \cite[Lemma 4.11]{MR4088505}.
\begin{lemma}\label{lem:dynk}
		Let the set $U\subset\subset D$ be open and Lipschitz.
If $\int_{\mR^d} |\phi (y)| (1\wedge \nu(y))\, \dy < \infty$ and $\phi \in C^2(\oU)$,
 then $L\phi$ is bounded on $\oU$ and
 for every $x\in \mR^d$,
		\begin{equation}\label{eq:A1}
		\mE^x\phi (X_{\tau_U}) - \phi (x) = \int_U G_U(x,y) L \phi (y) \,\dy,
		\end{equation}
where the integrals converge absolutely.
\end{lemma}
The following Hardy--Stein formula
extends \cite[Lemma 8]{MR3251822} and
\cite[Lemma 4.12]{MR4088505},
where
it was proved, for the fractional Laplacian
and $p>1$, and
for
unimodal operators $L$ and $p=2$, respectively.
	\begin{proposition}\label{prop:BDL-0}
		If $u\colon\mathbb R^d\to\mathbb R$ is $L$-harmonic in $D,$
$p >1 $, and $U \subset\subset D$ is open Lipschitz, then
		\begin{eqnarray}\label{eq:lemBDL}
		\mE^x |u(X_{\tau_U})|^p &=& |u(x)|^p + \int_U G_U(x,y)\pord F_p
(u(y),u(z)) \nu(y,z)\, \dz \dy,\quad x\in U .
		\end{eqnarray}
	\end{proposition}
\begin{proof}
As a guideline, the result follows by taking $\phi= |u|^p$ in the Dynkin formula \eqref{eq:A1}. We combine the methods of \cite{MR3251822} and \cite{MR4088505}. By \cite[Lemma 4.9]{MR4088505} if $u$ is harmonic in $D$, then $u\in C^2(D)$. Thus, in particular, $|u|^p$ is bounded in a neighborhood of $\overline{U}$.
Let $x\in U$. Consider the complementary cases:
\[ {\rm (i)}\, \int_{ U^c}|u(z)|^p
\nu(x,z)\,\dz=\infty, \quad \mbox{ or } \quad {\rm (ii)}\, \int_{U^c}|u(z)|^p
\nu(x,z)\,\dz<\infty.\]
Since $|u|^p$ is bounded in a neighborhood of $\overline{U}$,
 this dichotomy can be reformulated as
\[ {\rm (i)}\quad\mathbb E^x |u(X_{\tau_U})|^p=\infty, \quad \mbox{ or } \quad {\rm (ii)}\quad \mathbb E^x |u(X_{\tau_U})|^p<\infty,\]
see the end of the proof of \cite[Lemma 4.11]{MR4088505} and \eqref{eq:pjP}.

In case (i), we show that
 the right-hand side of \eqref{eq:lemBDL} is infinite as well.
Assume first that $|u|>0$ on a subset of $U$ of positive measure. Pick $y\in U$ satisfying $|u(y)|>0$, and let $A = \{z\in
 {U^c}: |u(z)| \geq (2+\sqrt{2})|u(y)|\}$.
Now, since $x,y\in U$ are fixed and {$\nu$ is positive, continuous, and satisfies \eqref{es:nuSc1}}, we have $\nu(x,z) \approx \nu(y,z)$ for $z\in U^c$. Therefore, by (i),
$$
\int_{ U^c}|u(z)|^p
\nu(y,z)\,\dz=\infty
$$
as well. Furthermore,
$$\int_{U^c\setminus A} |u(z)|^p \nu(y,z)\, \dz\approx\int_{U^c\setminus A} |u(z)|^p \nu(x,z)\, \dz \leq (2+\sqrt 2)^{{p}}|u(y)|^p \nu(x,U^c) < \infty,$$
and consequently we must have
\[\int_A |u(z)|^p \nu(y,z)\,{\rm d}z = \infty.\]
By the definition of $A$, for $z\in A$ we have
\begin{equation}\label{eq:sets}
(u(z) - u(y))^2 \geq \frac 12 u(z)^2 \quad \text{ and } \quad |u(z)| \geq |u(y)|.
\end{equation}
By \eqref{elem-ineq} and \eqref{eq:sets} we therefore obtain
\begin{align*}
\int_{\mR^d} F_p(u(y),u(z)) \nu(y,z)\, \dz &\approx \int_{\mR^d} (u(z) - u(y))^2(|u(y)|
 \vee |u(z)|)^{p-2} \nu(y,z)\, \dz\\
&\geq \int_A (u(z) - u(y))^2|u(z)|^{p-2}\nu(y,z) \, \dz
\geq \tfrac 12\int_A |u(z)|^p \nu(y,z)\, \dz =\infty.
\end{align*}
This is true for all points $y$ in a set of positive Lebesgue measure, which proves that the right-hand side of \eqref{eq:lemBDL} is infinite. If, on the other hand, $u\equiv 0$ in $U$, then $F_p(u(y),u(z)) = c|u(z)|^p$ for all $z\in \mR^d$, $y\in U$, and by (i)
the right-hand side of \eqref{eq:lemBDL} is infinite again.

We now consider the case (ii). Thus $\mathbb E^x|u(X_{\tau_U})|^p<\infty$ and the integrability condition of Lemma \ref{lem:dynk} is satisfied for $\phi = |u|^p$. We will first prove \eqref{eq:lemBDL} for $p\geq2$. Then $\phi$ is of class $C^2$ on $D$,
 so we are in a position to use Lemma \ref{lem:dynk} and we get
 \begin{equation}\label{eq:h-s}
 \mathbb E^x |u(X_{\tau_U})|^p= |u(x)|^p +\int_U G_U(x,y)
 L|u|^p(y)\,\dy,\quad x\in U.
 \end{equation}
The integral on the right-hand side is absolutely convergent. Furthermore,
since $u$ is $L$-harmonic,
\begin{eqnarray*}
L|u|^p(y)&=& L|u|^p(y)- p u(y)^{\langle p-1\rangle}Lu(y)\\
&=& \lim_{\epsilon\to 0+}\int_{|z-y|>\epsilon}
(|u(z)|^p- |u(y)|^p -p u(y)^{\langle p-1\rangle}(u(z)-u(y)))\nu(y,z)\,\dz\\
&=& \int_{\mathbb R^d} F_p(u(y),u(z))\nu(y,z)\,\dz\geq 0.
\end{eqnarray*}
 Inserting this to \eqref{eq:h-s} gives the statement.

 When $p\in (1,2)$, the function $\R \ni r\mapsto |r|^p$ is not twice differentiable, and
 the above argument needs to be modified. We work under the assumption (ii), and we follow the proof of \cite[Lemma 3]{MR3251822}.
Consider $\varepsilon\in \R$ and the function $\R^d\ni x\mapsto (x^2+\varepsilon^2)^{p/2}$. Let
\begin{equation}\label{eq:Feps}
F_p^{(\varepsilon)}(a,b) =
(b^2+\varepsilon^2)^{p/2}-(a^2+\varepsilon^2)^{p/2}
 - pa(a^2+\varepsilon^2)^{(p-2)/2}(b-a)\,, \quad a,b\in \R.
\end{equation}
Since $1<p<2$, by
\cite[Lemma 6]{MR3251822},
\begin{equation}
 \label{eq:ub}
0\le F_p^{(\varepsilon)}(a,b)\leq \frac{1}{p-1}F_p(a,b)\,,\qquad \varepsilon, a, b\in\R\,,
\end{equation}
Let $\varepsilon>0$. We note that $(u^2+\varepsilon^2)^{p/2}\in C^2(D)$. Also, the integrability condition in Lemma~\ref{lem:dynk} is satisfied for
$\phi=(u^2+\varepsilon^2)^{p/2}$ since it is satisfied for $\phi=|u|^p$ by (ii), and
\begin{equation}\label{eq:bb}
(u^2+\varepsilon^2)^{p/2} \leq (|u|+\varepsilon)^p\le 2^{p-1}( |u|^p + \varepsilon^p),
\end{equation}
see also \eqref{eq:Lmc}. Furthermore, $\mathbb E_{x}(u(X_{\tau_U})^2+\varepsilon^2)^{p/2}<\infty$.
As in the first part of the proof,
\begin{align}
L(u^2+\varepsilon^2)^{p/2}(y)
&=
L(u^2+\varepsilon^2)^{p/2}(y)
-p u(y)(u(y)^2+\varepsilon^2)^{(p-2)/2}Lu(y)\label{eq:subtr}\\
&=
\int_\Rd
F_p^{(\varepsilon)}(u(y),u(z))\nu(y,z)\,\dz,\nonumber
\end{align}
therefore by Lemma~\ref{lem:dynk},
\begin{equation}\label{eq:epsto0}
\mathbb E_{x}(u(X_{\tau_U})^2+\varepsilon^2)^{p/2}=(u(x)^2+\varepsilon^2)^{p/2}+\int_U G_U(x,y)
 \int_\Rd F_p^{(\varepsilon)}(u(y),u(z))\nu(y,z)\,\dz\dy.
\end{equation}
From the Dominated Convergence Theorem the left-hand side of \eqref{eq:epsto0} goes to $\mathbb E_x|u(X_{\tau_U})|^p<\infty$ as $\varepsilon\to 0^+$.
Of course, $F_p^{(\varepsilon)}(a,b)\to F_p(a,b)$ as $\varepsilon\to 0^+$.
Furthermore, by Fatou's
lemma and \eqref{eq:epsto0},
\begin{eqnarray*}
\int_U G_U(x,y)\int_{\mathbb R^d}F_p(u(y),u(z))\nu(y,z)\,\dz\dy
&\leq& \liminf_{\varepsilon\to 0^+}\int_U G_U(x,y)\int_{\mathbb R^d}F_p^{(\varepsilon)} (u(y),u(z))\nu(y,z)\,\dz\dy \\
&=& \mathbb E_x|u(X_{\tau_U})|^p-|u(x)|^p<\infty.
\end{eqnarray*}
By
\eqref{eq:ub} and the Dominated Convergence Theorem,
 we obtain \eqref{eq:lemBDL} for $p\in (1,2)$.
 \end{proof}

As a consequence, we obtain the
the Hardy--Stein identity for $D$, generalizing and strengthening \cite[(16)]{MR3251822} and \cite[Theorem 2.1]{MR4088505}.
\begin{proposition}\label{prop:hardy-stein} Let $p>1$ be given.
If $u$ is $L$-harmonic in $D$ and $x\in D$, then
		\begin{eqnarray}\label{eq:BDL}
\sup_{x\in U\subset\subset D} \mE^x |u(X_{\tau_U})|^p &=&
|u(x)|^p + \int_D G_D(x,y)\pord F_p(u(y),u(z)) \nu(y,z)\, \dz \dy.
		\end{eqnarray}
If $u$ is regular $L$-harmonic in $D$, then the left-hand side can be replaced with $\mE^x |u(X_{\tau_D})|^p$.
\end{proposition}
\begin{proof}
As noted in \cite[Remark 4.4]{MR4088505},
 $\{u(X_{\tau_U}), \, U\subset D\}$
 is a martingale ordered by the inclusion of open subsets of $D$. By domain monotonicity of the Green function and the nonnegativity of $F_p$, both sides of \eqref{eq:lemBDL} increase if $U$ increases.
Since every open set $U\subset\subset D$ is included in an open Lipschitz set $U\subset\subset D$, the supremum in \eqref{eq:BDL} may be taken over open Lipschitz sets $U\subset\subset D$. The first part of the statement follows from the monotone convergence theorem.

If additionally $u$ is regular harmonic, then
		\begin{eqnarray}\label{eq:BDLr}
\mE^x |u(X_{\tau_D})|^p &=&
|u(x)|^p + \int_D G_D(x,y)\pord F_p(u(y),u(z)) \nu(y,z)\, \dz \dy.
		\end{eqnarray}
This is delicate. Indeed, by \cite[Remark 4.4]{MR4088505}, the martingale $\{u(X_{\tau_U}), \, U\subset \subset D\}$ is closed by the integrable random variable $u(X_{\tau_D})$.
Therefore L\'evy's Martingale Convergence Theorem yields that $u(X_{\tau_U})$ converges almost surely, and in $L^1$ to a random variable $Z$, as $U\uparrow D$, and we have $Z = \mE^x[u(X_{\tau_D})|\sigma(\bigcup_{U\subset\subset D}\mathcal F_{\tau_U})]$, see, e.g., Dellacherie and Meyer \cite[Theorem 31 a,b, p. 26]{MR745449}. We claim that the $\sigma$-algebra
$\sigma(\bigcup_{U\subset\subset D}\mathcal F_{\tau_U})$ is equal to $\mathcal F_{\tau_D}$. Indeed, by Proposition 25.20 (i),(ii), and Proposition 25.19 (i),(ii) in Kallenberg \cite[p. 501]{MR1876169}, the filtration of $(X_t)$ is quasi-left continuous. Therefore $\tau_U$ increases to $\tau_D$ as $U$ increases to $D$, and our claim follows from Dellacherie and Meyer \cite[Theorem 83, p. 136]{MR521810}. Consequently, $Z = u(X_{\tau_D})$. Now, if $\sup_{x\in U\subset\subset D} \mE^x|u(X_{\tau_U})|^p < \infty$, then \cite[Theorem 31 c, p. 26]{MR745449} yields \eqref{eq:BDLr}.
Else, if the supremum is infinite, then $\mE^x |u(X_{\tau_D})|^p = \infty$ by Jensen's inequality, and \eqref{eq:BDLr} holds, too.
\end{proof}
We note in passing that the case $p=2$ of \eqref{eq:BDLr} was stated for less general sets $D$ in the first displayed formula following (5.2) in the proof of Theorem 2.3 in \cite{MR4088505}. Accordingly, the proof in \cite{MR4088505} was easier.

\section{The Douglas identity}\label{sec:mainthm}

We now present our main theorem. It is a counterpart of \eqref{eq:ident-l2}
with square increments of the function replaced by ``increments'' measured
in terms $F_p$ or $H_p.$
\begin{theorem}[\bf Douglas identity] \label{th:main}
Let $p>1$. Assume that the L\'{e}vy measure $\nu$ satisfies
{\bf (A1)} and {\bf (A2)}, $D\subset \mR^d$ is open, $D^c$ satisfies {\rm (VDC)}, and $|\partial D| = 0$.
\begin{enumerate}
\item[(i)]Let $g\colon D^c\to\R$ be such that
\(\hpd[g]<\infty.\)
Then $P_D[g]$ is well-defined and satisfies
\begin{equation}
\label{eq:ident-fp}
\hpd[g]=\epd[P_D[g]].
\end{equation}
\item[(ii)]Furthermore, if $u\colon \mR^d \to \mR$ satisfies $\epd[u] < \infty$, then $\hpd[u|_{D^c}] < \infty$.
\end{enumerate}
\end{theorem}
Here, as usual, $u|_{D^c}$ is the restriction of $u$ to $D^c$, but in what follows we will abbreviate:
$$
\hpd[u]:=\hpd[u|_{D^c}]
$$
and
$$
P_D[u|_{D^c}] = P_D[u].
$$

\begin{remark}\label{rem:pf1.8}
The more explicit expression of the Douglas identity
\eqref{eq:nsDi} stated in the Introduction follows from \eqref {eq:ident-fp},
\eqref{eq:Ep} and \eqref{eq:sHp}.
\end{remark}
To the best of our knowledge the present Douglas identities are completely new, and our approach is original.
The proof of Theorem~\ref{th:main} is given below in this section.

Recall the space $\vdp$, defined in \eqref{eq:vdp}, which is a natural domain of $\epd$, and the space $\xdp$, defined in \eqref{eq:xdp}, which is a natural domain of $\hpd$.
From Theorem~\ref{th:main} we immediately obtain the following trace and extension result in the nonquadratic setting.
\begin{corollary}\label{coro:extension}
Let ${\rm Ext }\, g=P_D[g]$, the Poisson extension, and ${\rm Tr}\, u=u|_{D^c}$, the restriction to $D^c$. Then ${\rm Ext}\colon \xdp\to \vdp$, ${\rm Tr}\colon \vdp \to \xdp$,
and
${\rm Tr}\,{\rm Ext}$ is the identity operator on $\xdp$.
\end{corollary}
It is well justified to call ${\rm Ext}$ the extension operator and ${\rm Tr}$ the trace operator for $\vdp$.

We next give the Douglas identity for the Poisson extension on $D$ and the form $\form{E}{\mR^d}{p}$ (with the integration over the whole of $\Rd\times\Rd$).
\begin{corollary}\label{cor:nnDia} If $P_D[|g|]<\infty$ on $D$, in particular if $\hpd[g] < \infty$, then
$$\form{E}{\mR^d}{p}[P_D[g]] = \tfrac 1p\iint\limits_{D^c\times D^c} F_p(g(z),g(w)) (\gamma_D(z,w)+\nu (z,w))\, \dz\dw.$$
\end{corollary}
We note that the kernel on the right-hand side of the above identity also appears in \cite[Theorem 5.6.3]{MR2849840} for $p=2$, but even the form $\form{E}{D}{2}$ and the Douglas identity of Theorem~\ref{th:main} with $p=2$ on full domain $\spc{V}{D}{2}$ do not appear in \cite{MR2849840}.

The proof of Theorem \ref{th:main} uses the following lemma,
 which asserts that the condition
 $\hpd[g]<\infty$ implies the finiteness of $P_D[|g|^p]$ and $P_D[|g|]$ on $D$.
\begin{lemma}\label{lem:poisint}
Suppose that $g\colon D^c\to \R$ satisfies $\hpd[g]<\infty.$
Then for every $x\in D$ we have
 $\int_{D^c}|g(z)|^p P_D(x,z)\,\dz<\infty$. In particular, the Poisson integral of $g$ is well-defined.
\end{lemma}
\begin{proof}
Denote $I=\int_{D^c}|g(z)|^p P_D(x,z)\,\dz$.
If $\hpd[g]<\infty$, then
\begin{eqnarray}\label{eq:form-fd-expanded}
&&\iint\limits_{D^c\times D^c}F_p(g(w),g(z))\gamma_D(w,z)\,\dw\dz\nonumber
\\ &=&\int_D\int_{D^c}\int_{D^c}F_p(g(w),g(z))
\nu(w,x)P_D(x,z)\,\dz\dw\dx<\infty.
\end{eqnarray}
Since $\nu > 0$, for
almost all (hence for some) pairs $(w,x)\in D^c\times D$ we get
\begin{equation}\label{eq:Ffin}
\int_{D^c} F_p(g(w),g(z))P_D(x,z)\,\dz
 <\infty.
\end{equation}
For the remainder of the proof, we only consider pairs $(w,x)$ satisfying the above condition.

 We will use different approaches for $p\geq 2$ and $p\in (1,2).$
Let $p\geq 2$. From \eqref{elem-ineq-2} we obtain
\[A:=\int_{D^c} (g(z)-g(w))^2|g(z)|^{p-2}P_{D}(x,z)\,\dz<\infty.\]

For $z\in D^c$, let $g_n(z)= -n\vee g(z) \wedge n$. Clearly $|g_n(z)|\leq |g(z)|$ and
 $|g_n(z)|\nearrow |g(z)|$ when $n\to \infty$. Since $|g_n(z)|\leq n, $
the integral $I_n:=\int_{D^c}|g_n(z)|^p P_D(x,z)\,\dz$ is finite. It is
also true that the increments of $g_n$ do not exceed those of $g$, that is
$|g_n(z)-g_n(w)|\leq |g(z)-g(w)|$. Consequently,
\begin{eqnarray*}
I_n
&= &
 \int_{D^c}g_n(z)^2 |g_n(z)|^{p-2} P_D(x,z)\,\dz\nonumber
 \\
 &\leq& 2\int_{D^c} (g_n(z)-g_n(w))^2|g_n(z)|^{p-2}P_D(x,z)
+ 2g_n(w)^2 \int_{D^c} |g_n(z)|^{p-2}P_D(x,z)\,\dz\\
&\leq & A
+ 2g(w)^2 \left(\int_{D^c} |g_n(z)|^{p}P_D(x,z)\,\dz\right)^{\frac{p-2}{p}}.\nonumber
\end{eqnarray*}
The last inequality is obvious for $p=2$, and follows from Jensen's inequality if $p>2$.
Thus,
\begin{equation}\label{eq:In-bound}
I_n\leq A + 2g(w)^2 (I_n)^{1-\frac{2}{p}},
\end{equation}
hence the sequence $(I_n)$ is bounded. By the Monotone Convergence Theorem,
$I
<\infty$.
By Jensen's inequality we also get $\int_{D^c}|g(z)| P_D(x,z)\,\dz<\infty$. By the Harnack inequality, the finiteness of the Poisson integral of $|g|$ or $|g|^p$ at any point $x\in D$ guarantees its finiteness at every point of $D$, see, e.g., \cite[Lemma 4.6]{MR4088505}, therefore the proof is finished for $p\geq 2$.

Now let $p\in (1,2)$. If $g\equiv 0$ a.e. on $D^c$, then the statement is trivial. Otherwise, pick $w\in D^c$ such that $0<|g(w)|<\infty$.
Let $
B=\{z\in D^c: |g(z)| > |g(w)|\}$.
We have
	$$\int_{D^c\setminus B} |g(z)|^{p} P_D(x,z)\, \dz \leq |g(w)|^p<\infty.$$
Using \eqref{elem-ineq} and \eqref{eq:Ffin} we get
\begin{align*}
	&\int_{B}|g(z)|^p P_D(x,z)\,\dz =
	\int_{B}g(z)^2 |g(z)|^{p-2} P_D(x,z)\,\dz
	\\
	&\leq 2\int_{B} (g(z)-g(w))^2|g(z)|^{p-2}P_D(x,z)\,\dz
	 + 2g(w)^2 \int_{B} |g(z)|^{p-2}P_D(x,z)\,\dz\\
	 &\approx
	 \int_B F_p(g(w),g(z))P_D(x,z)\, \dz+2|g(w)|^p<\infty.
\end{align*}
Thus, $P_D[|g|^p](x)<\infty$. The rest of the proof is the same as in the case $p\ge 2$.
\end{proof}

\begin{proof}[Proof of Theorem \ref{th:main}]
To prove (i) we let $\hpd[g]<\infty$ and we have \eqref{eq:form-fd-expanded}. Let $u = P_D[g]$. By \eqref{eq:exg}, $u=g$ on $D^c$. By Lemma \ref{lem:poisint}, $u$ is well-defined, and it is regular $L$-harmonic in $D$,
that is $\mathbb E^x[u(X_{\tau_D})]=u(x)$ for $x\in D$, cf. Definition~\ref{def:rh} and \eqref{eq:always-jump}. In particular, we have $\mE^x|u(X_{\tau_D})|<\infty$.

For $x\in D$ consider the integral
$\int_{D^c}F_p(u(w),u(z)) P_D(x,z)\,\dz$. By \eqref{eq:always-jump}, $P_D(x,z)$ is the density of the distribution of $X_{\tau_ D}$ under $\mathbb P^x$, hence $$
\int_{D^c}F_p(u(w),u(z)) P_D(x,z)\,\dz = \mathbb E^x[F_p(u(w), u(X_{\tau_D}))].
$$
By Lemma \ref{lem:aaa} (ii) applied to $a=u(w)$, $X=u(X_{\tau_D})$ and $\mathbb E=\mathbb E^x$, the above expression is equal to
\begin{equation}\label{eq:var}
F_p(u(w), \mathbb E^xu(X_{\tau_D}))+
\mathbb E^xF_p(u(x), u(X_{\tau_D}))= F_p(u(w),u(x))+\mathbb E^xF_p(u(x),u(X_{\tau_D})).
\end{equation}
By integrating the first term on the right-hand side of \eqref{eq:var} against $\nu(x,w)\,\dx\dw$ we obtain
\begin{equation}\label{eq:part-one}
\iint\limits_{D^c\times D} F_p(u(w),u(x))\nu(x,w)\,\dx\dw.
\end{equation}
For the second term in \eqref{eq:var} we use
Lemma \ref{lem:aaa} (i) and
Proposition \ref{prop:hardy-stein}:
\begin{eqnarray*}
\mathbb E^xF_p(u(x), u(X_{\tau_D}))
&=& \mathbb E^x|u(X_{\tau_D})|^p -|u(x)|^p
= \int_D G_D(x,y)\int_{\mathbb R^d}F_p(u(y),u(z))\nu(y,z)\,\dz\dy.
\end{eqnarray*}
We integrate the latter expression against $\nu(x,w)\,\dx\dw$. By Fubini--Tonelli, \eqref{eq:Pk} and \eqref{eq:always-jump},
\begin{align} \label{eq:part-two}
&\int_{D^c}\int_D \int_D \int_{\mathbb R^d}G_D(x,y)F_p(u(y),u(z))\nu(y,z)
\nu(x,w)\,\dz\dy\dx\dw\nonumber\\
= &\int_D\int_{\mathbb R^d} F_p(u(y),u(z))
\bigg(\int_{D^c}\bigg(\int_D G_D(x,y) \nu(x,w)\,\dx\bigg)\dw\bigg)\nu(y,z)\,\dz\dy
\nonumber\\
= &\int_D\int_{\mathbb R^d}F_p(u(y),u(z))
 \bigg(\int_{D^c} P_D(y,w)\,\dw\bigg)\nu(y,z)\,\dz\dy\nonumber\\
 = &\int_{D}\int_{\mathbb R^d} F_p(u(y),u(z))\nu(y,z)\,\dz\dy.
 \end{align}
Since the sum of \eqref{eq:part-one} and \eqref{eq:part-two} equals $p\epd[u]$, we obtain the Douglas identity.

We now prove (ii).
It is not obvious how to directly conclude that $\epd[u] < \infty$
implies
$P_D[|u|]<\infty$ on $D$,
thus we cannot apply
Lemma \ref{lem:aaa}. Instead we use another approach: by Lemma \ref{lem:quad}, $\epd[u] < \infty$ is equivalent to $\form{E}{D}{}[u^{\langle p/2\rangle}] < \infty$. By the trace theorem for $p=2$, see \cite[Theorem 2.3]{MR4088505}, $\form{H}{D}{} [u^{\langle p/2\rangle}] < \infty$. By Lemma \ref{lem:quad} we get (ii).
\end{proof}

\section{Douglas and Hardy--Stein identities with remainders}\label{sec:aug}
Throughout this section we assume that $D$ is bounded.
In the (quadratic) case $p=2$, under a mild additional assumption on $D, $ the Poisson integral $P_D[g]$ was shown to be the minimizer of the form $\form{E}{D}{}$ among all Borel functions with a fixed exterior condition $g\in\spc{X}{D}{}$ (see \cite[Proposition 5.4 and Theorem 5.5]{MR4088505}). This needs not be the case when $p\neq 2$, and in this section we give an example of $D$ and $g\in \xdp$ for which $P_D[g]$ does not minimize $\epd$ among functions in $\vdp$ equal to $g$ on $D^c$. However, $P_D[g]$ is always a \emph{quasiminimizer}, if we adopt the following definition:
\begin{definition} Let $K\geq 1.$
Function $u\in \vdp$ is a $K$-quasiminimizer of $\epd$, if $\form{E}{U}{p}[u] \leq K\form{E}{U}{p}[v]$ for every nonempty open set $U\ar{\subset} D$ \ar{satisfying {\rm (VDC)} and $|\partial U| = 0$}, and every $v\in \spc{V}{U}{p}$ equal to $u$ on $U^c$. We say that $u$ is a quasiminimizer if it is a $K$-quasiminimizer for some $K\in [1,\infty)$.
\end{definition}
The definition is inspired by the classical one given by Giaquinta and Giusti \cite[(5.26)]{MR666107}. To avoid technical complications and to make this digression short we require regular test sets $U$ above.
However, to be prudent
we note that the choice of admissible sets $U$ may affect the definition of quasiminimizers and should be carefully considered, cf. Giusti \cite[Example 6.5]{MR1962933}. In the classical PDEs, quasiminimizers display many regularity properties similar to minimizers, see, e.g., Adamowicz and Toivanen \cite{MR3373594}, DiBenedetto and Trudinger \cite{MR778976}, and Ziemer \cite{MR823124}. The main motivation for studying quasiminimizers is the fact that the solution of a complicated variational problem may be a quasiminimizer of a better understood functional see, e.g., \cite[Theorem 2.1]{MR666107}.

\begin{proposition}\label{prop:Q-min} Suppose that the assumptions of Theorem \ref{th:main} are satisfied, $D$ is bounded, and let $g\in \xdp.$
Then $P_D[g]$ is a $K$-quasiminimizer of $\epd$ with $K$ independent of $g$.
\end{proposition}

\begin{proof}
Fix a subset $U\subset D$ \ar{satisfying {\rm (VDC)} and $|\partial U| = 0$}, and let $v\in \spc{V}{U}{p}$ be equal to $u:=P_D[g]$ on $U^c$. According to
\eqref{eq:apEp2} we have $v^{\langle p/2\rangle}\in\spc{V}{U}{}$
and
\[\form{E}{U}{p}[v] \approx \form{E}{U}{}[v^{\langle p/2\rangle}],\]
with constants independent of $U$ and $v$.
Note that $v^{\langle p/2\rangle}$ agrees with $u^{\langle p/2\rangle}$ on $U^c$. Since
$U^c$ satisfies {\rm (VDC)}, by \cite[Proposition 5.4 and Theorem 5.5]{MR4088505},
\begin{equation}\label{eq:qm}
\form{E}{U}{}[v^{\langle p/2\rangle}] \geq
\form{E}{U}{}[P_U[u^{\langle p/2\rangle}]].
\end{equation}
By applying the Douglas identity for the set $U$, first with exponent $2$, and then with exponent $p$, and by \eqref{eq:apHp2} we get that the right-hand side of \eqref{eq:qm} is equal to
\begin{align*}
 \form{H}{U}{}[u^{\langle p/2\rangle}] \approx \form{H}{U}{p}[u] = \form{E}{U}{p}[P_U[u]] = \form{E}{U}{p}[u].
\end{align*}
In the last equality we use the identity $P_U[u]= u$, see \eqref{eq:always-jump}.
The proof is complete.
\end{proof}

To prove that Poisson integrals need not be minimizers,
we first extend the Hardy--Stein and Douglas identities to
functions that are not harmonic. The results are new even for $p=2$ and $\Delta^{\alpha/2}$.

\smallskip

Recall that $D$ is bounded, hence
$\mE^x\tau_D$ is bounded.
In what follows by $\lim\limits_{U\uparrow D}$ we denote the limit over an arbitrary ascending sequence of Lipschitz open sets $U_n \subset\subset D$ such that
$\bigcup_{n} U_n = D$.
Here is an extended version of the Hardy--Stein formula.
\begin{proposition}\label{prop:hsaug} Let $p> 1$ and assume that $\nu$ satisfies {\bf (A1)} and {\bf (A2)}. Let $u\colon \Rd\to \R$. If $u\in C^2(D)$ and $u$ and $Lu$ are bounded in $D$, then for every $x\in D$,
	\begin{align}
	\lim\limits_{U\uparrow D}\mathbb E^x |u(X_{\tau_U})|^p&= |u(x)|^p+ \int_D G_D(x,y)\int_{\mathbb R^d} F_p(u(y),u(z))\nu(y,z)\,{\rm d}z{\rm d}y\label{eq:new0}\\
	&+ p \int_D G_D(x,y)u(y)^{\langle p-1\rangle}Lu(y)\,{\rm d}y.\label{eq:new}
	\end{align}
If in addition $D^c$ satisfies {\rm (VDC)} and $|\partial D| = 0$, then
$\lim\limits_{U\uparrow D}\mathbb E^x |u(X_{\tau_U})|^p=\mathbb E^x |u(X_{\tau_D})|^p$.
\end{proposition}
\begin{proof}Let $x\in D$.
Since $u$, $Lu$, and $\mE^x \tau_D$ are bounded on $D$, by \eqref{eq:gdtd} we get that the integral in \eqref{eq:new} is finite. Therefore, using the arguments from the proof of Proposition \ref{prop:hardy-stein}, in what follows we may and do assume that $\int_{\mR^d} |u(x)|^p (1\wedge \nu(x))\, \dx < \infty$, because otherwise both sides of \eqref{eq:new0} are infinite.
With this in mind we first consider open Lipschitz $U\subset\subset D$ so large that $x\in U.$
	
Let $p\geq 2$.	Since $u\in C^2(D),$
 we get that $L|u|^p(x)$ and $\mathbb E^x|u(X_{\tau_{U}})|^p$ are finite for $x\in U$,
 and \eqref{eq:h-s} holds. Furthermore, since $Lu$ is finite in $D$, the following manipulations are justified for $y\in D$:
	\begin{align}
 L|u|^p(y) =&\, L|u|^p(y)-pu(y)^{\langle p-1\rangle}Lu(y)
	+pu(y)^{\langle p-1\rangle}Lu(y)\label{eq:pluszero}\\
		=& \lim\limits_{\epsilon \to 0^+}\int_{|z-y|>\epsilon} \big(|u(z)|^p - |u(y)|^p -pu(y)^{\langle p-1\rangle}(u(z)-u(y))\big)\nu(z,y)\,\dz\nonumber\\
	&+
	pu(y)^{\langle p-1\rangle}Lu(y)\nonumber \\
=&\int_{\mR^d}F_p(u(y),u(z))\nu(y,z)\, \dz + pu(y)^{\langle p-1\rangle}Lu(y).\nonumber
	\end{align}
Consequently, \eqref{eq:h-s} takes on the form
	\begin{align}
	\label{eq:hsU} \mE^x |u(X_{\tau_U})|^p = |u(x)|^p &+ \int_U G_U(x,y)\int_{\mR^d} F_p(u(y),u(z))\nu(y,z)\, \dz\dy \\
	&+ \int_U G_U(x,y) u(y)^{\langle p-1\rangle}Lu(y)\, \dy.\label{eq:hsU2}
	\end{align}
For clarity we note that the left-hand side of \eqref{eq:hsU} is finite and the integral in \eqref{eq:hsU2} is absolutely convergent, so the integral in \eqref{eq:hsU} is finite as well.

	For $p\in (1,2)$ we proceed as in the proof of Proposition \ref{prop:BDL-0}, that is, instead of $|u(x)|^p$ we consider $\varepsilon > 0$ and the function $x\mapsto (u(x)^2 + \varepsilon^2)^{p/2}$.
We obtain (cf. \eqref{eq:subtr} and \eqref{eq:pluszero}),
	\begin{align}
	\mE^x(u(X_{\tau_U})^2 + \varepsilon^2)^{p/2} = (u(x)^2 + \varepsilon^2)^{p/2} &+ \int_U G_U(x,y)\int_\Rd
	 F_p^{(\varepsilon)}(u(y),u(z))\nu(y,z)\,\dz\dy\label{eq:hsUeps}\\
	&+ p\int_U G_U(x,y)u(y)(u(y)^2 + \varepsilon^2)^{(p-2)/2}Lu(y)\, \dy.\label{eq:hsU2eps}
	\end{align} As in the proof of Proposition \ref{prop:hardy-stein},
the left-hand side tends to $\mE^x |u(X_{\tau_U})|^p$ as $\varepsilon \to 0^+$. Furthermore, since $Lu$ and $u$ are bounded in $D$, the integral in \eqref{eq:hsU2eps} converges to that in \eqref{eq:hsU2}. Then we apply Fatou's lemma and the Dominated Convergence Theorem to the integral on the right-hand side of \eqref{eq:hsUeps} and we obtain \eqref{eq:hsU} for $p\in(1,2)$, too.
	
We let $U\uparrow D$ in \eqref{eq:hsU}. By the boundedness of $u$ and $Lu$ in $D$, the integral in \eqref{eq:hsU2} tends to the one in \eqref{eq:new}, which is absolutely convergent. The integral on the right-hand side of \eqref{eq:hsU} converges to the one on the right-hand side of \eqref{eq:new0} by the domain monotonicity and the Monotone Convergence Theorem.
Since the limit on the right-hand side of \eqref{eq:new0} exists, the limit on the left-hand side must exist as well. This proves \eqref{eq:new0}.

If $D^c$ satisfies {\rm (VDC)} and $|\partial D| = 0$, then \eqref{eq:always-jump} holds true. Furthermore, we have
	\begin{align*}
	\mE^x|u(X_{\tau_U})|^p = \mE^x (|u(X_{\tau_U})|^p;\,\tau_U \neq \tau_D) + \mE^x(|u(X_{\tau_D})|^p;\, \tau_U = \tau_D).
	\end{align*}
	The first term on the right converges to $0$ by the boundedness of $u$ on $D$ and the fact that $\mathbb P^x(\tau_U\neq~\tau_D)$ decreases to $0$ as $U\uparrow D$ (see the remark preceding \eqref{eq:always-jump}; see also the proof of Lemma 17 in Bogdan \cite{MR1438304} and the proof of Lemma A.1 in \cite{MR4088505}).
	The second term converges to $\mE^x |u(X_{\tau_D})|^p$ by the Monotone Convergence Theorem.
Thus the left-hand side of \eqref{eq:hsU} tends to $\mE^x |u(X_{\tau_D})|^p$.
\end{proof}

We next provide a Douglas-type identity for a class of nonharmonic functions:
\begin{theorem}\label{th:Douglas2}
Suppose that the assumptions of Theorem \ref{th:main} hold with the addition that $D$ is bounded.
Let $u\colon \Rd\to \R$ be bounded, $u\in C^2(D)$, and $Lu$ be bounded in $D$. Then
	\begin{equation}\label{eq:Dnh}
	\epd[P_D[u]] = \epd[u] + A_D(u),
	\end{equation}
	where
	\[A_D(u)=\int_D u(x)^{\langle p-1\rangle}Lu(x)\,{\rm d}x
	+\int_D\int_{D^c} u(w)^{\langle p-1\rangle}(u(x)-P_D[u](x))\,\nu(w,x)\,{\rm d}w{\rm d}x.\]
\end{theorem}
\begin{proof}
Since $u$ is bounded on $\mR^d$, we have $\int_{\mR^d} |u(x)|(1\wedge \nu(x))\, \dx < \infty$.

	 Assume first that $\hpd[u]
	 < \infty$.
	From Theorem \ref{th:main} we have $$\epd[P_D[u]] = \hpd[u].$$
By \eqref{eq:rg2} and Fubini--Tonelli,
	$$p\hpd[u]= \int_D\int_{D^c}\int _{D^c}F_p(u(w),u(z))P_D(x,z)\nu(x,w)\,\dz {\rm d}w{\rm d}x.$$
	We apply Lemma \ref{lem:aaa} (iii) to $a=u(w)$, $b=u(x)$, with $w\in D^c$ and $x\in D$, $X=u(X_{\tau_D})$, and $\mathbb E=\mathbb E^x$. Note that $\mathbb EX=P_D[u](x)$. This yields:
	\begin{align*}
		&\int_{D^c}F_p(u(w),u(z))P_D(x,z)\,{\rm d}z \\= &\int_{D^c}F_p(u(x),u(z))P_D(x,z)\,{\rm d}z + F_p(u(w),u(x))
		 + (pu(w)^{\langle p-1\rangle}-pu(x)^{\langle p-1\rangle})(u(x)-P_D[u](x)).
	\end{align*}
	After integration, we obtain
	\begin{eqnarray*}
		p\hpd[u] &=& \int_D\int_{D^c}\int_{D^c}F_p(u(x),u(z))P_D(x,z)\nu(x,w)\,\dz{\rm d}w{\rm d}x \\
		&& + \int_D\int_{D^c} F_p(u(w),u(x))\nu(x,w)\,{\rm d}w{\rm d}x\\
		&& +\int_D\int_{D^c} (pu(w)^{\langle p-1\rangle}-pu(x)^{\langle p-1\rangle})(u(x)-P_D[u](x))\,\nu(x,w)\,{\rm d}w{\rm d}x\\
		&=:& A_1(u)+A_2(u)+A_3(u).
	\end{eqnarray*}
	Note that every term above is finite. Indeed, by the boundedness of $u$,
	\begin{align*}
	|A_3(u)| \lesssim \int_D\int_{D^c} |u(x) - P_D[u](x)|\nu(x,w)\, \dw\dx.
	\end{align*}
To prove that this is finite, let $v = u -P_D[u]$. We have $Lv = Lu = f\in L^\infty(D)$ and $v= 0$ on $D^c$.
Note that $v\in C^2(D)$ and $\int_{\mR^d} |v(x)|(1\wedge \nu(x))\, \dx < \infty$, cf. \cite[Lemma 3.6]{MR4088505}.
Let $U\subset\subset D$. By Lemma~\ref{lem:dynk},
	\begin{align*}
	\mE^x v(X_{\tau_U}) - v(x) = \int_U G_U(x,y) f(y)\, \dy, \quad x\in U.
	\end{align*}
	Since $u$ is bounded on $\mR^d$, we have $\mE^x u(X_{\tau_U}) \to \mE^x u(X_{\tau_D}) = P_D[u](x)$ as $U\uparrow D$, cf. the last part of the proof of Proposition \ref{prop:hsaug}. Hence, the boundedness of $f$, the domain monotonicity, and the Dominated Convergence Theorem yield
	\begin{equation*}
	v(x) = -\int_D G_D(x,y) f(y)\, \dy, \quad x\in D.
	\end{equation*}
	This allows us to further estimate $A_3$:
	\begin{align*}
	|A_3(u)| \lesssim \int_D\int_{D^c}\int_D G_D(x,y) \nu(w,x)\,\dy\dw\dx = \int_D\int_{D^c} P_D(y,w)\, \dw\dy = |D| <\infty.
	\end{align*}
	Since $A_1(u)$ and $A_2(u)$ are nonnegative, they must be finite as well, because $\hpd[u] < \infty$.
	We then have
	\begin{align*}
	\int_{D^c}F_p(u(x),u(z))P_D(x,z)\,{\rm d}z &= \mathbb E^x F_p(u(x),u(X_{\tau_D}))\\ &= \mathbb E^x |u(X_{\tau_D})|^p - |u(x)|^p - pu(x)^{\langle p-1\rangle}(P_D[u](x) - u(x)).
	\end{align*}
	Thus, by Proposition \ref{prop:hsaug} we obtain
	\begin{align}
	A_1(u) = A_4(u) &+ p\int_D\int_{D^c}\int_D G_D(x,y)u(y)^{\langle p-1\rangle}Lu(y)\nu(x,w)\,\dy\dw\dx \label{eq:augDoug}\\
	&-p\int_D\int_{D^c} u(x)^{\langle p-1\rangle}(P_D[u](x) - u(x))\nu(x,w)\, \dw\dx,\nonumber
	\end{align}
	where $A_4(u)$ is the integral in \eqref{eq:part-two}.
	Note that
$A_2(u) + A_4(u) = p\epd[u]$. Also,
all the expressions in \eqref{eq:augDoug} are finite, see the discussion of $A_3(u)$. To finish the proof of \eqref{eq:Dnh} in the case $\hpd[u] < \infty$, we simply note that $pA_D(u) = A_1(u) - A_4(u) + A_3(u)$.
	
The situation $\hpd[u] = \infty$ remains to be considered. Since $P_D[u]$ is bounded in $D$, by arguments similar to those in the estimates of $A_3(u)$ above, we prove that $A_D(u)$ is finite. Therefore by Theorem \ref{th:main} the identity \eqref{eq:Dnh} holds
with both sides infinite.
\end{proof}

Knowing the form of the remainder $A_D(u)$ in the Douglas identity \eqref{eq:Dnh}, we may provide an example which shows that Poisson integral need not be a minimizer of $\epd$ for $p\neq 2$; it is only a quasiminimizer by Proposition~\ref{prop:Q-min}.
\begin{example}[The Poisson extension need not be a minimizer for $p\neq 2$]
\rm Let $p>2$ and consider $0<R<R_1$ such that $D\subset\subset B_R$. Define $$g_n(z) = ((|z|-R)^{-1/(p-1)} \wedge n)\textbf{1}_{B_{R_1}\setminus B_R}(z).$$ Since each $g_n$ is bounded with support separated from $D$, we have $g_n\in \xdp\cap \spc{X}{D}{}$; see the discussion following Example 2.4 in \cite{MR4088505}. By \eqref{eq:pjP} there exists $c>0$ such that
\begin{equation}\label{eq:pdsmall}
P_D(x,z) \leq c,\quad x\in D,\, z\in B_{R_1}\setminus B_R.
\end{equation}
Furthermore, for every $U\subset\subset D$ there is $\epsilon >0$ such that
\begin{equation}\label{eq:pdbig}
P_D(x,z) \geq \epsilon,\quad x\in U,\, z\in B_{R_1}\setminus B_R.
\end{equation}
For $x\in D$ we let
$$u_n(x) = G_D[1](x) + P_D[g_n](x).$$
 Obviously $u_n$ are bounded on $\mR^d$. We will verify that $G_D[1] \in C^2(D)$. For this purpose we let $f$ be a smooth, compactly supported, nonnegative function equal to 1 on $D$.
 	By the Hunt's formula and Fubini--Tonelli we get
 \begin{equation}\label{eq:GHunt}
 G_D[f](x) = G_D[1](x) = \int_{\mR^d} G(x-y)f(y)\, \dy - \mE^x\int_{\mR^d}G(X_{\tau_D},y) f(y) \, \dy, \quad x\in \mR^d.
 \end{equation}
 Here $G$ is either the potential kernel or the compensated potential kernel of $(X_t)$; see Grzywny, Kassmann and Le\.{z}aj \cite[Appendix A]{2018arXiv180703676G} for details. In particular, by \cite[Corollary A.3]{2018arXiv180703676G} and \cite[Theorem 35.4]{MR1739520} $G$ is locally integrable, thus the first term in \eqref{eq:GHunt} is finite and smooth in $D$. Since the latter term in \eqref{eq:GHunt} is a harmonic function, we get that $G_D[1]\in C^2(D)$.
In particular, by \cite[Lemma 4.10]{MR4088505} and Dynkin \cite[Lemma 5.7]{MR0193671} we have $Lu_n = -1$ in $D$.
 We are now in a position to apply Theorem \ref{th:Douglas2}. Fix open $U\subset\subset D$. We get
\begin{align}
A_D(u_n) &= -\int_D u_n(x)^{p-1}\, \dx + \int_D\int_{D^c}u_n(w)^{p-1}G_D[1](x)\nu(x,w)\, \dw\dx\nonumber\\
&= \int_D (\mE^x u_n(X_{\tau_D})^{p-1} - (\mE^x u_n(X_{\tau_D})+ G_D[1](x))^{p-1})\,\dx = \int_U + \int_{D\setminus U}.\label{eq:ADz}
\end{align}
We claim that $A_D(u_n) > 0$ for large $n$. Indeed, recall that $G_D[1](x)=\mathbb E^x \tau_D$ is bounded.
Since the integrals $\int_{D^c} g_n(x)\, \dx$ are bounded, by \eqref{eq:pdsmall}
there is $M>0$ such that $\mE^x u_n(X_{\tau_D}) < M$ for every $x\in D$ and $n\in \mathbb{N}$. Therefore the integral $\int_{D\setminus U}$ in \eqref{eq:ADz} is bounded from below, independently of~$n$.
Note that $\int_{D^c}g_n(x)^{p-1}\dx\to \infty$ as $n\to \infty$.
Thus, by \eqref{eq:pdbig} we obtain that
$\int_U\to \infty$ in \eqref{eq:ADz} as $n\to \infty$. Hence, for sufficiently large $n$ we get that $A_D(u_n) > 0$, which proves that $\epd[P_D[u_n]] > \epd [u_n]$ for some $n$, as needed. The case $p\in (1,2)$ may be handled similarly, by using $g_n(z) = ((|z|-R)^{-1} \wedge n)\textbf{1}_{B_{R_1}\setminus B_R}(z)$ and $u_n = P_D[g_n] - G_D[1]$.
\end{example}
\ar{\section{Applications to Dirichlet-to-Neumann map}\label{sec:app}}
\ar{
	In this section we adopt the assumptions of Theorem~\ref{th:main}. In addition, we assume that the set $D$ is bounded and $p\in [2,\infty)$. We define the nonlocal normal derivative as an analogue of the fractional version of Dipierro, Ros-Oton, and Valdinoci  \cite[(1.2)]{MR3651008}, see also Vondra\v{c}ek \cite{MR4245573}:
	\begin{align}\label{e.dnd}	
	\mathcal{N}f(z) = \int_D (f(z) - f(x))\nu(x,z)\, \dx.
	\end{align}
Note that the increments of $f$ are integrated on $D$, but the integral is evaluated for $z\in D^c$, if convergent.
For instance, if $f\in L^1(\mR^d,1\wedge \nu)$, then $\mathcal{N}f\in L^1_{loc}(D^c)$.}

\ar{Assume that $g\in\xdp$. Then $u=P_D[g]$ solves the Dirichlet problem \eqref{eq:Dp}. By definition, the Dirichlet-to-Neumann operator $\DN$ maps the exterior condition $g$ to the nonlocal normal derivative $h := \mathcal{N}u$. So, $u$ solves the Neumann problem
	\begin{align*}
	\left\{
	\begin{array}{ll}
	Lu=0 & \mbox{ in } D,\\
	\mathcal{N}u=h & \mbox{ on } D^c,
	\end{array}\right.
	\end{align*}
and $\DN := \mathcal{N}\circ P_D$ on $\xdp$. In fact, for almost every $z\in D^c$,
	\begin{align}
	\DN g(z)= \mathcal{N} u(z) &= \int_{D} (u(z) - u(x))\nu(x,z)\, \dx\label{eq:dtndef}\\ &= \int_{D}\int_{D^c} (u(z) - u(w))P_D(x,w)\nu(x,z)\, \dw \dx\nonumber \\ &= \int_{D^c} (u(z) - u(w))\gamma_D(z,w)\, \dw\nonumber\\ &= \int_{D^c} (g(z) - g(w))\gamma_D(z,w)\, \dw,\label{eq:explicitdtn}
	\end{align}
	where we have used the definition of $\gamma_D$, the fact that $u=P_D[g]$, and the Fubini--Tonelli theorem (justified by the estimates in the proof of Proposition~\ref{prop:weightedlp}).
	For $z\in {\rm Int}(D^c)=\Rd\setminus \overline{D}$ we let 
\begin{align}\label{e.dm}
m(z) := \int_{D^c} \gamma_D(w,z)\, \dw = \int_D\nu(x,z)\, \dx<\infty.
\end{align}
 For example, for $L = \Delta^{\alpha/2}$ and (bounded) $D$ of class $C^{1,1}$, with $\delta_D(z) := d(z,D)$ we have
	\begin{align*}
	m(z) = c_{d,\alpha}\int_D |z-x|^{-d-\alpha}\, \dx \approx \begin{cases}
	\delta_D(z)^{-\alpha},\quad &\delta_D(z)\leq 1,\\
	\delta_D(z)^{-d-\alpha},\quad &\delta_D(z)>1.
	\end{cases}
	\end{align*}
	}
\ar{Back to general $L$, we note that sharp estimates of $\gamma_D$ are known for bounded $C^{1,1}$ domains and the half-space, see \cite[Theorem~2.6 and 6.1]{MR4088505}.
We next define the \emph{normalized} Dirichlet-to-Neumann operator, for $g\in \xdp$ and $a.e.$ $z\in D^c$,
	\begin{align}\label{e.nDNo}
	\widetilde{\DN}g(z) = \frac{\DN g(z)}{m(z)} = \int_{D^c}(g(z) - g(w))\frac{\gamma_D(z,w)}{m(z)}\, \dz.
	\end{align}}	
	
\ar{In what follows we give several results for the Dirichlet-to-Neumann operator on $L^p$. In particular, we show that $\DN$ is well-defined: $\xdp\to L^p(D^c,m^{1-p})$ and $\widetilde{\DN}$ is bounded on  $L^p(D^c,m)$. We also relate the form $\hpd$ to the operator $\DN$ in \eqref{e.roH}.
	\begin{proposition}\label{prop:weightedlp}
		Assume that $g\in \xdp$. Then $\DN g\in L^p(D^c,m^{1-p})$ and $\widetilde{\DN} g\in L^p(D^c,m)$. Furthermore, there exists a constant $C$, independent of $g$, such that
		\begin{align*}
		\|\DN g\|_{L^p(D^c,m^{1-p})}^p = \|\widetilde{\DN}g\|_{L^p(D^c,m)}^p \leq C\hpd[g].
		\end{align*}
	\end{proposition}
	\begin{proof}
		Using \eqref{eq:explicitdtn} and Jensen's inequality we get
		\begin{align*}
		\int_{D^c} |\DN g(z)|^p m(z)^{1-p}\, \dz =  \int_{D^c} \bigg(\frac{|\DN g(z)|}{m(z)}\bigg)^p m(z)\, \dz \leq \int_{D^c}\int_{D^c} |g(w) - g(z)|^p\gamma_D(z,w)\, \dz \dw.
		\end{align*}
		Since $p\geq 2$, we have $|a-b|^p \leq (a-b)^2(|a|^{p-2} + |b|^{p-2})$. So, by \eqref{e.cFH}, 
		\begin{align}\label{e.ogbH}
 \int_{D^c}\int_{D^c} |g(w) - g(z)|^p\gamma_D(z,w)\, \dz \dw	
		\lesssim \hpd[g]<\infty,
		\end{align}
which ends the proof.
	\end{proof}
	\begin{proposition}\label{prop:weightedlpxd}
		If $g\in L^p(D^c,m)$, then $g\in \xdp$ and there is $C>0$, independent of $g$, such that
		\begin{align*}
		\hpd[g] \leq C\|g\|_{L^p(D^c,m)}^p.
		\end{align*}
	\end{proposition}
	\begin{proof}
		Following \cite[Remark~2.37]{2022arXiv220406793F}, we let $\widetilde{g}$ be the function $g$ extended to $D$ by 0. Then,
		\begin{align}\nonumber
		\epd[\widetilde{g}] &= \tfrac 1p\iint\limits_{(D^c\times D^c)^c} F_p(\widetilde{g}(z),\widetilde{g}(w)) \nu(z,w)\, \dz\dw = \int_D\int_{D^c}|g(z)|^p \nu(z,w)\, \dw\dz\\ &= \int_{D^c} |g(z)|^p m(z)\, \dz < \infty.\label{eq:lpm}
		\end{align}
		In particular, $\widetilde{g}\in \vdp$. By Proposition~\ref{prop:Q-min} we get that there exists a constant $C$, independent of $g$, such that $\epd[P_D[g]] \leq C\epd[\widetilde{g}]$. Using this, Theorem~\ref{th:main}, and \eqref{eq:lpm}, we find that
		\begin{align*}
		\hpd[g] = \epd[P_D[g]] \leq C\epd[\widetilde{g}] = C\int_{D^c} |g(z)|^p m(z)\, \dz,
		\end{align*}
		which proves the result.
	\end{proof}
	\begin{corollary}\label{cor:dtn}
		The normalized Dirichlet-to-Neumann map $\widetilde{\DN}$ is bounded on $L^p(D^c,m)$.
	\end{corollary}
	The following is an analogue of the formula \eqref{e:fag} below.
	\begin{proposition}
		Let $f\in L^p(D^c,m)$ and $g\in \xdp$. Then $\int_{D^c} |\DN g(z)||f(z)|^{p-1}\, \dz < \infty$ and
		\begin{align}\label{e.rDN}
		\int_{D^c} \DN g(z) f(z)^{\langle p-1\rangle}\, \dz = \tfrac 12\int_{D^c}\int_{D^c} (g(z) - g(w))(f(z)^{\langle p-1\rangle} - f(w)^{\langle p-1\rangle})\gamma_D(z,w)\, \dz\dw.
		\end{align}
		Furthermore, if $g\in L^p(D^c,m)$, then
		\begin{align}\label{e.roH}
		\int_{D^c} \DN g(z) g(z)^{\langle p-1\rangle}\, \dz = \int_{D^c} \widetilde{\DN}g(z)g(z)^{\langle p-1\rangle} m(z)\, \dz = \hpd[g].
		\end{align}
	\end{proposition}
	\begin{proof}
		By H\"older's inequality with exponents $p$ and $p' = \frac{p}{p-1}$, and by  Proposition~\ref{prop:weightedlp},
		\begin{align*}
		&\int_{D^c} |\DN g(z)||f(z)|^{p-1}\, \dz = \int_{D^c} |\DN g(z)| m(z)^{\frac{1-p}{p}} m(z)^{\frac{p-1}{p}} |f(z)|^{p-1}\, \dz\\
		\leq &\bigg(\int_{D^c} |\DN g(z)|^p m(z)^{1-p}\, \dz\bigg)^{\frac 1p} \bigg(\int_{D^c} |f(z)|^p m(z)\, \dz\bigg)^{\frac {p-1}p}<\infty.
		\end{align*}
It suffices to prove \eqref{e.rDN}.
By the symmetry of $\gamma_D$, 
		\begin{align*}
		\int_{D^c} \DN g(z) f(z)^{\langle p-1\rangle}\, \dz &= \int_{D^c}\int_{D^c} (g(z) - g(w))f(z)^{\langle p-1\rangle}\gamma_D(z,w)\, \dw\dz\\
		&=\int_{D^c}\int_{D^c} (g(z) - g(w))f(z)^{\langle p-1\rangle}\gamma_D(z,w)\, \dz\dw\\
		&=\int_{D^c}\int_{D^c} (g(w) - g(z))f(w)^{\langle p-1\rangle}\gamma_D(z,w)\, \dw\dz.
		\end{align*}
The above application of the Fubini--Tonelli theorem is justified by using H\"older's inequality with exponents $p$ and $p/(p-1)$, and \eqref{e.ogbH}; see also \eqref{e.dm}.
The first and the last lines above yield \eqref{e.rDN}.
	\end{proof}}

	\ar{Let us discuss related results for $p=2$. In \cite[Proposition~3.2]{MR4245573}, Vondra\v{c}ek shows that the normalized Dirichlet-to-Neumann operator map is bounded on $L^2(D^c,m)$; our Corollary~\ref{cor:dtn} extends this result to $L^p$. As observed by Foghem and Kassmann \cite[Remark~2.37]{2022arXiv220406793F}, the space $L^2(D^c,m)$ can be smaller than the trace space $\spc{X}{D}{}$. In \cite[Section~4.4]{2022arXiv220406793F}, the authors investigate the Dirichlet-to-Neumann operator for the equation $Lu = \lambda u + f$, where $\lambda\in \mR$ is not a Dirichlet eigenvalue of $L$ in $D$. They prove the boundedness of the Dirichlet-to-Neumann operator from the trace space into its dual. If we let $\DN_{FK}$ be the Dirichlet-to-Neumann operator defined in \cite{2022arXiv220406793F} for $\lambda = 0$ and $f=0$, then using our Douglas identity and $\widetilde{v} = u_g = P_D[g]$ in \cite[Definition~4.18]{2022arXiv220406793F}, for $g\in\spc{X}{D}{}$ we get
\begin{align*}
\langle \DN_{FK} g, g\rangle = \form{E}{D}{} [P_D[g]] = \form{H}{D}{} [g] = \frac 12\int_{D^c}\int_{D^c} (g(z) - g(w))^2\gamma_D(z,w)\,\dz\dw.
\end{align*}
Here $\langle \cdot,\cdot \rangle$ is the pairing between $\spc{X}{D}{}$ and its dual, see \cite[Section~2.6]{2022arXiv220406793F}. Then, by polarization, 
\begin{align}\label{eq:polarized}
\langle \DN_{FK} g_1,g_2\rangle  = \form{H}{D}{} (g_1,g_2) = \frac 12 \int_{D^c}\int_{D^c} (g_1(z) - g_1(w))(g_2(z) - g_2(w))\gamma_D(z,w)\,\dz\dw,
\end{align}
for $g_1,g_2\in \spc{X}{D}{}$.
Both \eqref{eq:explicitdtn} and \eqref{eq:polarized} give explicit integral representations for the Dirichlet-to-Neumann operator, which are more direct than \eqref{eq:dtndef}. They were not stated in \cite{2022arXiv220406793F,MR4245573}, although similar formulas appear in \cite[Section~3 and (4.2)]{MR4245573} and the author of \cite{MR4245573} was probably aware of the explicit versions.}

\ar{On an informal level, \eqref{eq:explicitdtn} and \eqref{eq:polarized} mean that the Dirichlet-to-Neumann map is the negative of the L\'evy-type operator on $D^c$ with jump kernel $\gamma_D$, and $\spc{H}{D}{}$ is the corresponding Dirichlet form. 
Despite being smaller than $\spc{X}{D}{2}$, the space $L^2(D^c,m)$, used by Vondra\v{c}ek, is suitable for studying the (negative of the) normalized Dirichlet-to-Neumann operator $\widetilde{\DN}$ as a generator of a Markov process on $D^c$. In fact, 
$-\widetilde{\DN}$ is the generator of the so-called trace process, see \cite[(4.2)]{MR4245573}. In this connection, the reader may compare \eqref{e.roH}, $\DN$ and $\form{H}{D}{p}$ with \eqref{e:fag}, $-L$ and $\form{E}{D}{p}$;
see also \cite[Lemma 7]{MR4372148} for a detailed discussion of $\form{E}{\Rd}{p}$ for $L=\Delta^{\alpha/2}$.
}

\section{Further discussion}\label{sec:discussion}
As usual, $D$ is a nonempty open set in $\Rd$. We define
\begin{equation}\label{eq:sp-wu}
\wdp=\bigg\{u\colon\R^d\to\R\ \bigg|\ \iint\limits_{\R^d\times\R^d\setminus
D^c\times D^c}|u(x)-u(y)|^p\nu(x,y)\,{\rm d}x{\rm d}y<\infty\bigg\},
\end{equation}
and
\[\ydp=\bigg\{g\colon D^c\to \R\ \bigg|\ \iint\limits_{D^c\times D^c}
|g(w)-g(z)|^p\gamma_D(w,z)\,{\rm d}w{\rm d}z<\infty\bigg\}.\]
\begin{proposition}\label{th:second}
If $p\geq 2$ then \eqref{eq:th61} holds true
under the assumptions on $D$ and $\nu$ from Theorem \ref{th:main}, and the Poisson extension acts from $\ydp$ to $\wdp$.

\end{proposition}
\begin{proof}
Assume that $g\in \ydp,$ i.e., the right-hand side of \eqref{eq:th61} is finite. By a simple modification of the proof of \cite[Lemma 4.6]{MR4088505} we get that $g\in L^p(D^c, P_D(x,z)\,\dz)$ for every $x\in D$, in particular the Poisson integral
$P_D[g](x)$ converges absolutely.
By
\eqref{eq:rg2}, the right-hand side
of \eqref{eq:th61} equals
\[\int_{D^c}\int_{D^c}\int_D |g(w)-g(z)|^p \nu(w,x)P_D(x,z)\,\dx\dw\dz.\]
We use Fubini--Tonelli and consider the integral
\[\int_{D^c}|g(w)-g(z)|^pP_D(x,z)\,\dz = \mathbb E^x\left|u(X_{\tau_D})- g(w)\right|^p.\]
By Lemma
\ref{lem:elem-p} we get that for $x\in D$ and $w\in D^c$,
\[\mathbb E^x\left|u(X_{\tau_D})-g(w)\right|^p \approx \mathbb E^x|u(X_{\tau_D})-u(x)|^p
+\left|u(x)-g(w)\right|^p \geq \mathbb E^x|u(X_{\tau_D})-u(x)|^p.\]
We apply Proposition \ref{prop:hardy-stein},
to
$\widetilde{u}(z):=u(z)-u(x).$
It is $L$-harmonic on $D$ and $\widetilde u(x)=0$, therefore
\[\mathbb E^x\left|u(X_{\tau_D})-u(x)\right|^p = \int_DG_D(x,y)
\int_{\mathbb R^d}F_p(\widetilde {u}(y), \widetilde{u}(z))\nu(z,y)\,\dz\dy.\]
For $p\neq 2$ it is not true that $F_p(a+t,b+t)$ is comparable with $F_p(a,b)$, but since $p\geq 2$, by Lemma \ref{lem:p-incr} we have $F_p(a+t,b+t)\geq c |a+t-b-t|^p = c|a-b|^p$.
It follows that
\[F_p(\widetilde u(y), \widetilde u(z))\gtrsim
 |u(y)-u(z)|^p,\]
and thus
\[\mathbb E^x\left|u(X_{\tau_D})-g(w)\right|^p \gtrsim \int_DG_D(x,y)
\int_{\mathbb R^d} |u(y)-u(z)|^p\nu(z,y)\,\dz\dy.\]
We integrate the inequality on $D^c\times D$
against $\nu(w,x)\,\dw\dx$ as in \eqref{eq:part-two}, and the right-hand side is
\begin{equation*}
\int_D\int_{\mR^d} |u(x)-u(y)|^p\nu(x,y) \,\dx \dy \geq \tfrac 12\!\!\! \iint\limits_{\mR^d\times \mR^d \setminus (D^c\times D^c)} |u(x) - u(y)|^p \nu(x,y)\,\dx\dy.
\end{equation*}
The result follows.
\end{proof}
We remark that in general
\eqref{eq:th61}
fails
for $p\in (1,2)$;
see Lemma~\ref{lem:pl2z} and Example \ref{ex:th61}.

In the remainder of this section we compare $\wdp$ and $\vdp$, see \eqref{eq:vdp}, by using $C_c^\infty(\mR^d)$.
\begin{lemma}\label{lem:sf}
For every $p>1$ we have $C_c^\infty(\mR^d)\subseteq \spc{V}{\mR^d}{p}\subseteq \vdp$.
\end{lemma}
\begin{proof}
The inclusion $\spc{V}{\mR^d}{p}\subseteq \vdp$ follows from the definition. To prove that $C_c^\infty(\mR^d)\subseteq \spc{V}{\mR^d}{p}$, we let $\phi\in C_c^\infty(\mR^d)$.
We have
$$
|\phi(x+z) + \phi(x-z) - 2\phi(x)|\leq M(1\wedge |z|^2), \quad x, z\in \Rd.
$$
It follows that $L\phi$ is bounded on $\mR^d$, cf. \eqref{eq:L-def} and \eqref{eq:Lmc}.
Thus,
\begin{equation}\label{eq:scLfi}
\int_\Rd |\phi(x)|^{p-1}|L\phi(x)|\,\dx<\infty.
\end{equation}
Furthermore, by the Dominated Convergence Theorem and the symmetry of $\nu$,
	\begin{align*}
	\int_\Rd \phi(x)^{\langle p-1\rangle}L\phi(x)\,\dx &= \tfrac 12\int_{\mR^d} \phi(x)^{\langle p-1\rangle}\lim\limits_{\epsilon \to 0^+} \int_{|z|>\epsilon}(\phi(x+z) + \phi(x-z) - 2\phi(x))\nu(z)\,\dz\dx \\
	&=\lim\limits_{\epsilon\to 0^+} \int_{\mR^d}\int_{|z|>\epsilon}\phi(x)^{\langle p-1\rangle}(\phi(x+z) - \phi(x))\nu(z)\,\dz\dx.
	\end{align*}
	By Fubini's theorem, the substitutions $z\to -z$ and $x\to x+z$, and the symmetry of $\nu$,
	\begin{align*}
	&\int_{\mR^d}\int_{|z|>\epsilon}\phi(x)^{\langle p-1\rangle}(\phi(x+z) - \phi(x))\nu(z)\,\dz\dx\\
	= &\int_{\mR^d}\int_{|z|>\epsilon}\phi(x+z)^{\langle p-1\rangle}(\phi(x) - \phi(x+z))\nu(z)\,\dz\dx\\
	= &-\tfrac12 \int_{|z|>\epsilon}\int_{\mR^d}(\phi(x+z)^{\langle p-1\rangle}-\phi(x)^{\langle p-1\rangle})(\phi(x+z) - \phi(x))\,\dx\,\nu(z)\,\dz
	\end{align*}
for every $\epsilon > 0$.
By \eqref{eq:fsp}, the Monotone Convergence Theorem and the above,
	\begin{align}
\form{E}{\mR^d}{p}[\phi]&=
\tfrac12\int_{\Rd}\int_{\mR^d}(\phi(x+z)^{\langle p-1\rangle} - \phi(x)^{\langle p-1\rangle})(\phi(x+z) - \phi(x))\nu(z)\,\dx\dz \nonumber
\\
&=-\int_\Rd \phi(x)^{\langle p-1\rangle}L\phi(x)\,\dx.\label{e:fag}
\end{align}
The result follows from \eqref{eq:scLfi} and \eqref{eq:vdp}.	
\end{proof}
The inclusion $C_c^\infty(\mR^d)\subseteq \vdp$ indicates
that the Sobolev--Bregman spaces will be useful in variational problems posed in $L^p.$

The situation with the spaces $\wdp$ is more complicated. While for $p\geq 2$ we have a result similar to that of Lemma \ref{lem:sf}, for $p\in (1,2)$ it is not so.
More precisely, we have the following two lemmas:
\begin{lemma}\label{lem:pg2au}
For $p\geq 2$ we have $C_c^\infty(\mR^d) \subseteq \spc{W}{\mR^d}{p}\subseteq \wdp$.
\end{lemma}
\begin{proof}
For $\phi\in C_c^\infty(\mR^d)$ let $K = {\rm supp}\, \phi$. Then we have $|\phi(x) - \phi(y)| = 0$ on $K^c\times K^c$ and $$|\phi(x) - \phi(y)|^p \lesssim 1\wedge |x-y|^p \leq 1\wedge |x-y|^2,\quad x,y\in\mR^d\times\mR^d\setminus K^c\times K^c.$$
It follows that $\phi\in
\spc{W}{\mR^d}{p}$.
The inclusion $\spc{W}{\mR^d}{p}\subseteq \wdp$ is clear from the definition of the spaces.
\end{proof}

\begin{lemma}\label{lem:pl2z} Let $p\in (1,2)$ and assume that for some $r>0$ we have $\nu(y) \gtrsim |y|^{-d-p}$ for $|y|<r$.
 If $u\in \wdp$ has compact support in $\mR^d$ and vanishes on $D^c$, then $u\equiv 0$.
\end{lemma}
Results of this type are well-known for the spaces with integration over $D\times D$, where $D$ is connected.
Brezis \cite[Proposition 2]{MR1942116} shows that any measurable function must be constant in this case; a simpler proof of this fact was given by De Marco, Mariconda and Solimini \cite[Theorem~4.1]{MR2426913}. Lemma \ref{lem:pl2z} follows by taking $\Omega = \mR^d$ in the aforementioned results, but we present a different proof. Such facts also hold true in the context of metric spaces, see, e.g., Pietruska-Pa\l{}uba \cite{MR2102059}.
We will see in the proof of Lemma~\ref{lem:pl2z} that the result reduces to that with $D=\Rd$.
\begin{proof}[Proof of Lemma \ref{lem:pl2z}]
We may assume that $u$ is bounded, because the $p$-increments of $(0\vee u)\wedge 1$ do not exceed those of $u$. Thus, since $u$ is compactly supported, we get that
$u\in L^p(\mR^d)\cap L^2(\mR^d)$. Let
	\begin{equation*}
	\widehat{u}(\xi) = \int_{\mR^d} u(x) e^{-2\pi i \xi x}\, \dx, \quad \xi\in\mR^d.
	\end{equation*}
The Hausdorff--Young inequality asserts that for
$u\in L^p(\mR^d)$ we have
	\begin{equation}\|u\|_{p} \geq \|\widehat{u}\|_{p'},\label{eq:H-Y}
	\end{equation} where $p'=\frac p{p-1}$,
see, e.g., Grafakos \cite[Proposition 2.2.16]{MR2445437}.
	We
	estimate the left-hand side of \eqref{eq:th61} by using \eqref{eq:H-Y}:
	\begin{align*}
	\iint\limits_{\mR^d\times\mR^d\setminus D^c\times D^c} |u(x) - u(y)|^p\nu(x,y)\,\dx\dy
	&=\int_{\mR^d}\int_{\mR^d} |u(x) - u(x+y)|^p\nu(y)\,\dx\dy\\[-10pt]
	&\geq \int_{\mR^d}\bigg(\int_{\mR^d} |(u(\cdot) - u(\cdot + y))^{\wedge}(\xi)|^{p'}\, {\rm d}\xi\bigg)^{\frac p{p'}}\nu(y)\, \dy\\
	&=\int_{\mR^d}\bigg(\int_{\mR^d} |1 - e^{-2\pi i \xi y}|^{p'} |\widehat{u}(\xi)|^{p'}\, {\rm d}\xi\bigg)^{\frac p{p'}}\nu(y)\,\dy.
	\end{align*}
	By \eqref{eq:H-Y}, $|\widehat{u}(\xi)|^{p'} \, {\rm d}\xi$ is a finite measure on $\mR^d$. As we have $p/p' < 1$, by Jensen and Fubini--Tonelli,
	\begin{align*}
	\int_{\mR^d}\bigg(\int_{\mR^d} |1 - e^{-2\pi i \xi y}|^{p'} |\widehat{u}(\xi)|^{p'}\, {\rm d}\xi\bigg)^{\frac p{p'}}\nu(y)\,\dy &\gtrsim \int_{\mR^d} \nu(y)\int_{\mR^d}|1- e^{-2\pi i \xi y}|^p |\widehat{u}(\xi)|^{p'}\, {\rm d}\xi \dy\\
	&=\int_{\mR^d}|\widehat{u}(\xi)|^{p'}\int_{\mR^d}|1 - e^{2\pi i \xi y}|^p\nu(y)\, \dy {\rm d}\xi.
	\end{align*}
	Since $|1 - e^{2\pi i\xi y}| \geq |\sin 2\pi \xi y|$ and $\nu(y) \gtrsim |y|^{-d-p}$ for small $|y|$, the integral is infinite, unless $u= 0$ a.e. in $\mR^d$.
\end{proof}
As a comment
to Lemmas~\ref{lem:sf} and \ref{lem:pl2z}
we recall that $\vdp$ is defined in terms of $F_p$. When $a$ is close to $b$ then, regardless of $p>1$,
the Bregman divergence $F_p(a,b)$ is of order $(b-a)^2$ rather than $|b-a|^p$. Thus $\vdp$ agrees with the L\'evy measure condition \eqref{eq:Lmc} better than $\wdp$ does.

The following example indicates that the scale of linear spaces $\wdp$ may not be suitable for analysis of harmonic functions when $p\leq 2$:
\begin{example}\label{ex:th61}
{\rm
Let $\nu$ and $p$ be as in Lemma \ref{lem:pl2z}. Let $B=B(0,1)$ and assume that $D$ is bounded and $\dist (D,B) > 0$.
Then there is
$g\in \ydp$ such that $u:=P_D[g]\notin \wdp$, i.e.,
\begin{equation}\label{eq:notinw}
\iint\limits_{\mR^d\times \mR^d \setminus D^c\times D^c} |u(x) - u(y)|^p \nu(x,y)\,\dx\dy = \infty.
\end{equation}

Let $g(z)=\textbf{1}_{B}(z)$ for $z\in D^c$. Then $g\in \ydp$, cf. the arguments following \cite[Example 2.4]{MR4088505}. Clearly, $u$ is bounded in $D$.
	By the positivity of $P_D$ \cite[Lemma 2.2]{MR3729529}, $u(x) > 0$ for every $x\in D$.
Of course, $B$, $D^c\setminus B = B^c\setminus D$ and $D$ form a partition of $\Rd$. Therefore their Cartesian products partition $\Rd\times \Rd$; in fact also $B^c\times B^c$ and $\mR^d\times \mR^d \setminus D^c\times D^c$ (see below).
	Since $u$ vanishes on $D^c\setminus B$, $u(x)-u(y)$ vanishes on $(D^c\setminus B)\times (D^c\setminus B)$. It follows that
	\begin{equation}\label{eq:Bc}\int_{B^c} \int_{B^c} |u(x) - u(y)|^p \nu(x,y)\,\dx\dy \leq \iint\limits_{\mR^d\times \mR^d \setminus D^c\times D^c} |u(x) - u(y)|^p \nu(x,y)\,\dx\dy.
	\end{equation}
Define $\widetilde{u}=u$ on $B^c$ and $\widetilde{u}=0$ on $B$. Then, $\widetilde u=u$ on $D$ and $\widetilde u = 0$ on $D^c$, and
	\begin{align*}
	&\iint\limits_{\mR^d\times \mR^d\setminus D^c\times D^c}|\widetilde{u}(x) - \widetilde{u}(y)|^p \nu(x,y)\, \dx\dy = \int_D\int_D +\ \int_D\int_{D^c\setminus B}+\ \int_{D^c\setminus B}\int_D +\ \int_B\int_D+\ \int_D\int_B\\
	&=\int_{B^c}\int_{B^c} |u(x) - u(y)|^p \nu(x,y)\, \dx\dy+ 2\int_D|u(y)|^p\int_B\nu(x,y)\, \dx\dy.
	\end{align*}
	By the boundedness of $u$, the boundedness of $D$ and the separation of $D$ and $B$, the last integral is finite. Furthermore, since $\widetilde{u}$ is not constant and vanishes on $D^c$, the left-hand side is infinite by Lemma \ref{lem:pl2z}. Therefore the left-hand side of \eqref{eq:Bc} is infinite, which yields \eqref{eq:notinw}.}
\end{example}

\end{document}